\documentclass[12pt]{amsart}
\usepackage{txfonts}      
\usepackage{amssymb}
\usepackage{eucal}
\usepackage{amsmath}
\usepackage{amscd}
\usepackage{xcolor}
\usepackage{multicol}
\usepackage[all]{xy}           
\usepackage{graphicx}
\usepackage{color}
\usepackage{colordvi}
\usepackage{xspace}
\usepackage{tikz}
\usetikzlibrary{positioning, arrows.meta}
\usepackage{makecell}
\usepackage{appendix}
\usepackage{amsthm}
\usepackage[misc]{ifsym}
\usepackage{mathrsfs} 
\usepackage{xypic}
\usepackage{extarrows}

\usepackage{ifpdf}
\ifpdf
\usepackage[colorlinks,final,backref=page,hyperindex]{hyperref}
\else
\usepackage[colorlinks,final,backref=page,hyperindex,hypertex]{hyperref}
\fi

\usepackage[active]{srcltx} 

\begin{document}
	\newtheorem{theorem}{Theorem}[section]
	\newtheorem{lemma}[theorem]{Lemma}
	\newtheorem{corollary}[theorem]{Corollary}
	\newtheorem{proposition}[theorem]{Proposition}
	\newtheorem{question}[theorem]{Question}
	\theoremstyle{definition}
	\newtheorem{definition}[theorem]{Definition}
	\newtheorem{example}[theorem]{Example}
	\newtheorem{remark}[theorem]{Remark}
	\newtheorem{pdef}[theorem]{Proposition-Definition}
	\newtheorem{condition}[theorem]{Condition}
	\renewcommand{\labelenumi}{{\rm(\alph{enumi})}}
	\renewcommand{\theenumi}{\alph{enumi}}
	\baselineskip=14pt

	\newcommand {\emptycomment}[1]{} 
	
	\newcommand{\nc}{\newcommand}
	\newcommand{\delete}[1]{}
	\newcommand{\bfk}{\mathbf{k}}
	\newcommand{\bfm}{\mathbf{m}}
	\newcommand{\id}{\textrm{id} }
	\newcommand{\ml}{\mathcal{L}}
	\newcommand{\mr}{\mathcal{R}}
	\newcommand{\N}{\mathcal{N}}
	\newcommand{\mt}{\mathcal{T}}
	\newcommand{\m}{\mathrm{m}}
	\newcommand{\we}{\widetilde }
	\newcommand{\dlie}{\mathsf{DLie}}
	\newcommand{\D}{\mathsf{D}}
	\newcommand{\Bij}{\mathsf{Bij}}
	\newcommand{\C}{\mathsf{C}}
	\newcommand{\fraks}{\mathfrak{S}}
	\newcommand{\smc}{\mathfrak{S}{\text-mod}_{\ch_\bfk}}
	\newcommand{\sbmc}{\mathfrak{S}{\text-bimod}_{\ch_\bfk}}
	\newcommand{\ch}{\mathsf{Ch}}
	\nc{\g}{\mathfrak{g} }
	\newcommand{\en}{\operatorname{End}}
	\newcommand{\Hom}{\operatorname{Hom}}
	\newcommand{\Func}{\operatorname{Func}}
	\newcommand{\Ind}{\operatorname{Ind}}
	\newcommand{\Aut}{\operatorname{Aut}}
	\newcommand{\mc}{\mathrm{C}}
	\nc{\lc}{\{\!\{  }
	\nc{\rc}{ \}\!\} }
	\nc{\lb}{[\![  }
	\nc{\rb}{ ]\!] }
	\nc{\ad}{\textrm{ad}}
	\nc{\bfr}{\mathbf{r}}
	\nc{\bfs}{\mathbf{s}}
	\nc{\bfu}{\mathbf{u}}
	\nc{\xzy}{\color{red}}
	
	\newcommand{\cm}[1]{\textcolor{red}{\underline{CM:}#1 }}

	\font\cyr=wncyr10
	\title{Controlling graded Lie algebras and cohomologies of double Lie algebras}
	
	\author{Zhongyin Xu}
	\address{Chern Institute of Mathematics \& LPMC, Nankai University,
		Tianjin, 300071, China}
	\email{zy\_xu@mail.nankai.edu.cn}
	\author{Chengming Bai}
	\address{Chern Institute of Mathematics \& LPMC, Nankai University, Tianjin 300071, PR China}
	\email{baicm@nankai.edu.cn}
	
	\begin{abstract}
		
		We construct a graded Lie algebra structure on the space of cyclically skew-symmetric cochains  whose Maurer-Cartan elements characterize double Lie algebra structures. Twisting by a fixed double Lie bracket yields a cochain complex that is isomorphic to the positive arity part of the complex  introduced by Fairon and Valeri after reindexing and degreewise sign adjustment.
		Using the Koszul resolution of the double Lie properad, we identify the resulting differential graded Lie algebra
		with  the  convolution differential graded Lie algebra on the deformation complex of the fixed double Lie algebra. 
		We also construct a graded Lie algebra governing skew-symmetric Rota-Baxter operators on  a finite-dimensional symmetric Frobenius algebra. For a finite-dimensional vector space $V$ and $A=\en(V)$, we lift the known correspondence between double Lie algebra structures on $V$ and skew-symmetric Rota-Baxter operators on $A$ to an isomorphism
		of their governing graded Lie algebras. After twisting by 
		corresponding Maurer-Cartan elements, this isomorphism becomes an
		isomorphism of differential graded Lie algebras and hence induces
		an isomorphism between the associated cohomology groups. We
		further prove that the cohomology of a skew-symmetric Rota-Baxter
		operator on a symmetric Frobenius algebra can be seen as the
		cyclic cohomology of its descendent associative algebra. Thus
		the cohomology of a double Lie algebra is identified with the
		cyclic cohomology of the descendent associative algebra
		associated with the skew-symmetric Rota-Baxter operator on the
		matrix algebra.

	\end{abstract}

	\maketitle
	\tableofcontents
	\section{Introduction}
	The notion of  double Poisson algebras  was introduced by Van den Bergh as a framework for noncommutative Poisson geometry  \cite{VdB},    under the Kontsevich-Rosenberg principle: if $A$ is a double Poisson algebra, then the associated affine representation schemes $\operatorname{Rep}_n(A)$ have (classical) Poisson structures.  They are closely related to   noncommutative geometry of quiver algebras  \cite{CBEG07,PVW08}, pre-Calabi-Yau algebras and homotopy double Poisson structures  \cite{IKV21,QW26}, the Poisson geometry of representation theory  \cite{VdB,VdW08}, and  noncommutative integrable systems  \cite{Arth15,DKV15}.
	
	A double Lie algebra is obtained by retaining from a double Poisson algebra the cyclic skew-symmetry and the double Jacobi identity, while forgetting the associative product and the double Leibniz rule. This notion was considered explicitly by     Schedler in connection with the associative Yang-Baxter    equation  \cite{S09}.  Moreover, every double Lie bracket on a vector space $V$ extends  uniquely, by the double Leibniz rule, to a double Poisson bracket on the    tensor algebra $T(V)$.  Therefore, double Lie algebras provide generator-level data for a natural class of double Poisson structures on free associative algebras.

	Fairon and Valeri  introduced a cochain complex and the associated cohomology for double Lie  algebras in  \cite[Section 4.1.4]{FV25}. In positive degrees, their cochains are linear maps $V^{\otimes n}\rightarrow V^{\otimes n}$ ($n\geq 1$)
	satisfying  cyclic skew-symmetry, and its differential is defined by insertion formulas involving a fixed double Lie bracket.
	In  classical deformation theory, the cochain complexes of certain classes of algebras often carry  a graded Lie algebra structure.  For example, after the standard degree shift, the Hochschild cochain complex
	of an associative algebra  with coefficients in the algebra itself   carries the Lie bracket of Gerstenhaber  \cite{G63},  while the Chevalley-Eilenberg cochain complex of a Lie algebra with adjoint coefficients carries
	the Lie bracket of Nijenhuis-Richardson \cite{NR67}. These examples suggest seeking an analogous explicit graded Lie
	structure on the cochain complex of a double Lie algebra.
	Moreover,  the convolution formalism in properadic deformation theory  provides   a conceptual framework for constructing graded Lie algebra structures on cochain complexes \cite{MV09,MV092}.  Leray proved that the properad governing double Lie algebras is Koszul \cite{Leray20}.
	Together with the results of Merkulov and Vallette,  there is a convolution differential graded (dg) Lie algebra structure on the deformation complex of a double Lie algebra via its Koszul resolution.
	So it is natural to relate this construction to the aforementioned dg Lie algebra on the double Lie algebra cochain   complex.
	This leads to our first question.
	\begin{question}\label{q1}
		Can one construct an explicit dg Lie algebra whose underlying graded vector space is the space of cyclically skew-symmetric cochains and differential  is  determined by a fixed  double Lie bracket?
		Can the resulting dg Lie algebra be identified with the convolution dg Lie algebra?
	\end{question}

	The notion of a Rota-Baxter operator originated in Baxter's work
	\cite{Bax} and was subsequently developed by Rota and  others
	\cite{R}.  It has important connections with, among other
	subjects,  renormalization in quantum field theory and operadic
	splitting  procedures  \cite{CK2,BBGN}. A systematic deformation
	theory of $\mathcal O$-operators on Lie algebras, including
	Rota-Baxter operators as a special case, was initiated in
	\cite{TBGS}, and its associative counterpart was subsequently
	developed in  \cite{Das20}.
	
	Our second question concerns a concrete realization of the aforementioned cochain complex of a double Lie algebra motivated by the following  correspondence. For a finite-dimensional vector space $V$, double Lie brackets on $V$ correspond to skew-symmetric Rota-Baxter operators on $A=\en(V)$, equipped with the trace pairing  \cite{GK18,ORS13,S09}.
	Therefore, it is natural to ask whether this correspondence  extends  from algebraic structures to the controlling
	graded Lie algebras.
	There is also a cohomological reason to investigate this extension. The cohomology of an $\mathcal {O}$-operator $T:V\rightarrow A$ on a Lie algebra  $A$ associated to a representation $V$ can be identified with
	the Chevalley-Eilenberg cohomology of  the induced Lie algebra $V$ with coefficients in $A$  \cite{TBGS} and the cohomology of
	$\mathcal O$-operators on associative algebras can be seen as the Hochschild cohomology of the induced associative algebra with a suitable   coefficient bimodule \cite{Das20}.
	However, the full Rota-Baxter cochain complex does not impose skew-symmetry. For a symmetric Frobenius algebra, the relevant issue is how
	the Frobenius pairing translates this constraint into a condition on scalar-valued cochains. These observations lead to a single comparison
	problem.
	\begin{question}\label{q2}
		Does the correspondence between double Lie brackets and
		skew-symmetric Rota-Baxter operators extend to an isomorphism of  their controlling graded Lie algebras and hence induce an isomorphism of their cohomology groups?  Can the cohomology of a skew-symmetric Rota-Baxter operator
		$\mr$ on a finite-dimensional symmetric Frobenius algebra $A$ be identified with a classical cohomology theory of its descendent associative algebra $A_{\mr}$?
		If so, when $A=\en(V)$, how does this identification relate the
		cohomology of $\mr$ as well as $A_{\mr}$ to the cohomology of the
		corresponding double Lie algebra structure on $V$?
	\end{question}

	The purpose of this paper is to answer Questions \ref{q1}-\ref{q2}.  The main results  are summarized as follows.
	\begin{enumerate}
		\item We construct a graded   Lie algebra $(\mathcal{C}_{dL}^\bullet (V),[\cdot,\cdot]_{dL})$ whose  Maurer-Cartan elements are precisely the  double Lie algebra structures on $V$; see Theorem \ref{th-dgla} and Corollary \ref{cor-mc-dla}. A fixed double Lie bracket $\mu\in\mathcal{C}^1_{dL}(V)$ determines the differential $d_\mu=[\mu,\cdot]_{dL}$ and hence a cohomology theory for the double Lie algebra $(V,\mu)$; see Corollary \ref{cor-mc-dla} and Definition \ref{def-dL-cochain}.
		
		Using the Koszul resolution of the double Lie properad, we identify $( \mathcal{C}_{dL}^\bullet(V),[\cdot,\cdot]_{dL},d_\mu)$ with the convolution dg Lie algebra $( \mathsf{Hom}_{\fraks}^\bullet (\D^{\text{\textexclamdown}},\en_V),[\cdot,\cdot]_{\mathrm{aug} } ,[\alpha_\mu,\cdot]_{\mathrm{aug}} )$; see Theorem \ref{th-properadic-deformation-complex}. Consequently, Corollary \ref{cor-dcc-dlc} gives an isomorphism of cohomology groups
		$$
		H^\bullet ( \mathsf{Hom}_\fraks^\bullet( \D^{\text {\textexclamdown} }, \en_V ), d_{\alpha_\mu} ) \cong H_{dL}^\bullet(V,\mu).
		$$
		This gives the double Lie algebra cochain complex and its cohomology a properadic interpretation and answers  Question \ref{q1}.
		\item  For a  finite-dimensional  symmetric Frobenius algebra $(A,\langle\cdot,\cdot\rangle)$, we construct a graded vector subspace  $\mc_{rb}^\bullet (A) \subset \bigoplus_{n\geq 0}\Hom(A^{\otimes n},A)$ and prove it  is closed under $[\cdot,\cdot]_{rb}$. Its Maurer-Cartan elements are precisely the skew-symmetric Rota-Baxter operators on $A$; see Theorem \ref{th-dgla-ssrb}  and Corollary \ref{cor-mc-ssrb}.
		When $A=\en(V)$, the structure correspondence extends to an isomorphism of graded Lie algebras
		$$
		\Psi:(\mathcal{C}_{dL}^{\bullet}(V),[\cdot,\cdot]_{dL})\xlongrightarrow{\cong}   (\mc_{rb}^{\bullet}(A),[\cdot,\cdot]_{rb}) .
		$$
		Therefore, it gives an isomorphism of dg Lie algebras after twisting by double Lie bracket $\mu$ and  skew-symmetric Rota-Baxter operator $\mr=\Psi_1(\mu)$ , and  then  gives an isomorphism of cohomology groups
		$$
		H_{dL}^{\bullet}(V,\mu)\cong    H_{rb,\mathrm{cyc}}^{\bullet}(A,\mr).
		$$ See Theorem \ref{th-dl-rb-dgla-isomorphism}  and Corollary \ref{cor-dl-rb-cohomology} .
		
		We show that the cyclic Rota-Baxter cochain complex $ (\mc_{rb}^\bullet(A), d_\mr )$ is isomorphic to the cyclic complex $(\mc_\tau^\bullet (A_\mr),d_C)$. Consequently, the cohomology of the skew-symmetric Rota-Baxter operator $\mr$ on $A$ is naturally isomorphic to the cyclic cohomology of the descendent associative algebra $A_\mr$:
		$$
		H_{rb,\mathrm{cyc}}^{p}(A,\mr)\cong HC_{\tau}^{p}(A_{\mr}), \qquad  p\geq0.
		$$
		In particular, if $A=\en(V)$ and $\mr$ corresponds to a double Lie bracket $\mu$ on $V$, then
		$$
		H_{dL}^{p}(V,\mu)\cong H_{rb,\mathrm{cyc}}^{p}(\en(V),\mr)\cong HC_{\tau}^{p} ((\en(V))_{\mr} ),\qquad p\geq0.
		$$
		See Proposition \ref{prop-cyclic-RB-cyclic-Hochschild} and Theorem \ref{th-dl-cyclic-Hochschild}.
		These comparisons answer Question \ref{q2}.
	\end{enumerate}

	The relationships established above can be summarized by the following  diagram:
	\begin{equation*}
		\resizebox{\textwidth}{!}{$
			
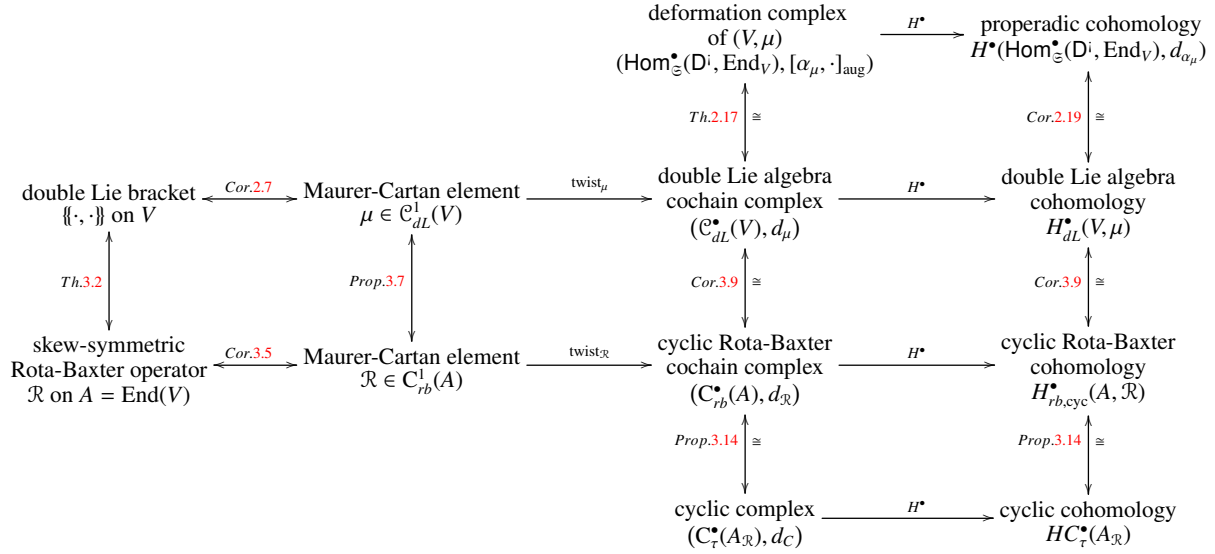
\xymatrix@C=1.5cm@R=1.4cm{
				& &\txt{ deformation complex\\of   $(V,\mu)$ \\ $( \mathsf{Hom}_{\fraks}^{\bullet}
					(\D^{\text{\textexclamdown}},\en_V),
					[\alpha_\mu,\cdot]_{\mathrm{aug}} )$ } \ar[r]^{\scriptstyle H^{\bullet}} \ar@{<->}[d]^{\scriptstyle\cong}_{\scriptstyle Th. \ref{th-properadic-deformation-complex} }&\txt{ properadic  cohomology\\ $ H^\bullet ( \mathsf{Hom}_\fraks^\bullet( \D^{\text {\textexclamdown} }, \en_V),d_{\alpha_\mu}) $ }  \ar@{<->}[d]^{\scriptstyle\cong}_{\scriptstyle Cor. \ref{cor-dcc-dlc}} \\
				\txt{double Lie bracket\\
					$\lc\cdot,\cdot\rc$ on $V$}
				\ar@{<->}[d]_{\scriptstyle Th. \ref{th-dl-rb}}
				\ar@{<->}[r]^{\scriptstyle Cor. \ref{cor-mc-dla}\quad}
				&
				\txt{Maurer-Cartan element\\
					$\mu\in\mathcal{C}_{dL}^{1}(V)$}
				\ar@{<->}[d]_{\scriptstyle Prop. \ref{prop-dl-rb-cochain-identification}}
				\ar[r]^{\scriptstyle\quad\mathrm{twist}_{\mu}}
				&
				\txt{double Lie algebra\\
					cochain complex\\
					$\bigl(\mathcal{C}_{dL}^{\bullet}(V),d_{\mu}\bigr)$}
				\ar@{<->}[d]^{\scriptstyle\cong}_{\scriptstyle Cor. \ref{cor-dl-rb-cohomology}}
				\ar[r]^{\scriptstyle H^{\bullet}}
				&
				\txt{double Lie algebra\\ cohomology\\
					$H_{dL}^{\bullet}(V,\mu)$}
				\ar@{<->}[d]^{\scriptstyle\cong}_{\scriptstyle Cor. \ref{cor-dl-rb-cohomology}}
				\\
				\txt{skew-symmetric\\
					Rota-Baxter operator\\
					$\mr$ on $A=\en(V)$}
				\ar@{<->}[r]^{\scriptstyle Cor. \ref{cor-mc-ssrb}\quad}
				&
				\txt{Maurer-Cartan element\\
					$\mr\in\mc_{rb}^{1}(A)$}
				\ar[r]^{\scriptstyle\quad \mathrm{twist}_{\mr}}
				&
				\txt{cyclic Rota-Baxter \\cochain complex\\
					$\bigl(\mc_{rb}^{\bullet}(A),d_{\mr}\bigr)$}
				\ar@{<->}[d]^{\scriptstyle\cong}_{\scriptstyle Prop. \ref{prop-cyclic-RB-cyclic-Hochschild}}
				\ar[r]^{\scriptstyle H^{\bullet}}
				&
				\txt{cyclic Rota-Baxter \\ cohomology \\$H_{rb,\mathrm{cyc}}^{\bullet}(A,\mr)$}
				\ar@{<->}[d]^{\scriptstyle\cong}_{\scriptstyle   Prop. \ref{prop-cyclic-RB-cyclic-Hochschild} }
				\\
				& &
				\txt{cyclic complex\\$\bigl(\mc_{\tau}^{\bullet}(A_{\mr}),d_C\bigr)$}
				\ar[r]^{\scriptstyle H^{\bullet}}
				&
				\txt{cyclic cohomology \\$HC_{\tau}^{\bullet}(A_{\mr})$}
			}
			$}
	\end{equation*}

	The paper is organized as follows.  In Section \ref{sec-dgla-dla}, we  construct the graded Lie algebra whose Maurer-Cartan elements characterize double Lie algebra structures, give rise to a double Lie algebra cochain complex by twisting a fixed double Lie bracket and establish the  properadic interpretation of the double Lie algebra cochain complex and the resulting cohomology.
	In Section \ref{sec-dgla-rb}, we construct the graded Lie  algebra whose Maurer-Cartan elements are precisely  skew-symmetric Rota-Baxter operators on a symmetric Frobenius algebra. We subsequently specialize to the matrix algebra and establish the  graded Lie algebra and cochain-complex isomorphisms with the double Lie algebra side.
	Finally, we identify double Lie algebra cohomology with the cyclic  cohomology of the descendent associative algebra associated with the skew-symmetric Rota-Baxter operator on the matrix algebra.

	\textbf{Notation}:
	Throughout this paper, let  $\bf k$ be a field of characteristic zero. All tensor products over ${\bf k}$ are denoted by $\otimes$. We denote the identity map by $\id$. All vector spaces $V$ and algebras $A$ over ${\bf k}$   are assumed to be finite-dimensional   unless otherwise stated, even though many results still hold in the infinite-dimensional  cases.
	For $n\geq 1$, let $\sigma_n:V^{\otimes n}\rightarrow V^{\otimes n}$ denote the cyclic
	permutation determined by
	$$
	(a_1\otimes \cdots \otimes a_n)^{\sigma_n} =\sigma_n(a_1\otimes \cdots \otimes a_n)  =a_n\otimes a_1\otimes \cdots \otimes a_{n-1},\qquad a_1,\dots,a_n\in V.
	$$
	When the tensor power is clear from the context, we simply write $\sigma$.

	\section{Differential graded Lie algebras associated with double Lie algebras}\label{sec-dgla-dla}
	We first  construct an explicit graded Lie algebra whose
	Maurer-Cartan elements are precisely the double Lie algebra
	structures on a given vector space. It  follows that a given
	double Lie bracket gives rise to a differential on this graded Lie
	algebra, thereby providing the cohomology of the double Lie
	algebra.  Finally, we identify this dg Lie algebra with the  convolution dg Lie algebra associated with the Koszul resolution of the double Lie properad.
	
	\subsection{    The graded Lie algebra governing double Lie algebra structures}
	\begin{definition}
		A  \textbf{double Lie algebra} is a vector space  $V$ equipped with a double bracket $\lc \cdot,\cdot\rc:V\otimes V\rightarrow V\otimes V$ satisfying the following identities: 
		\begin{align}
			&\tag{skewsymmetry} \label{skewsymmetry}\lc a,b \rc=-\lc b,a \rc^\sigma,\\
			&\tag{Jacobi identity}  \label{Jacobi identity} \lc a,\lc b,c \rc \rc_L+\lc b,\lc c,a \rc \rc_L^\sigma+\lc c,\lc a,b \rc \rc_L^{\sigma^2}=0 , \qquad a,b,c\in V.
		\end{align}
		(Here we use the notation $\lc a, b\otimes c\rc_L=\lc a,b\rc\otimes c $). 
	\end{definition}
	
	\begin{definition}
		Let $(\mathfrak{g}=\bigoplus_{k\in\mathbb{Z}} \mathfrak{g}^k,[\cdot,\cdot],d)$ be a \textbf{differential graded Lie algebra} (dg Lie algebra). A degree $1$ element $x\in \mathfrak{g}^1$ is called a \textbf{Maurer-Cartan element} of $\mathfrak{g}$ if it satisfies  the Maurer-Cartan equation
		\begin{equation*}
			dx+\frac{1}{2}[x,x]=0.
		\end{equation*}
		A graded Lie algebra is a dg Lie algebra with $d=0$.
	\end{definition}
	Let $V$ be a vector space.
	For $p\geq0$, set $$
	\mathcal{C}^p_{dL}(V):=\{ P\in \en(V^{\otimes (p+1)})=\operatorname{Hom}(V^{\otimes (p+1)},V^{\otimes (p+1)})| \sigma_{p+1}P\sigma_{p+1}^{-1}=(-1)^pP\}
	$$
	and define $\mathcal{C}^\bullet_{dL}(V)=\bigoplus_{p\geq0} \mathcal{C}^p_{dL}(V)  $. We assume that the degree of an element in  $\mathcal{C}^p_{dL}(V)$ is $p$.
	
	We define the $\square$-product of $P\in  \mathcal{C}^p_{dL}(V)$ and $Q\in \mathcal{C}^q_{dL}(V)$ by
	\begin{equation}
		(P\square Q) :=\sum_{s=0}^{p+q}(-1)^{(p+q)s}\sigma_{p+q+1}^s\circ (P\diamond Q)\circ \sigma_{p+q+1}^{-s},
	\end{equation}
	where $ P\diamond Q=    (P\otimes\id^{\otimes q})\circ(\id^{\otimes p}\otimes Q)$.
	\begin{lemma}
		For $P\in  \mathcal{C}^p_{dL}(V)$ and $Q\in \mathcal{C}^q_{dL}(V)$, one has $P\square Q\in  \mathcal{C}^{p+q}_{dL}(V)$.
	\end{lemma}
	\begin{proof}
		For $P\in  \mathcal{C}^p_{dL}(V)$ and $Q\in \mathcal{C}^q_{dL}(V)$, we have
		\begin{align*}
			&\sigma_{p+q+1}\circ(P\square Q)\circ\sigma_{p+q+1}^{-1}\\
			=&  \sum_{s=0}^{p+q} (-1)^{(p+q)s}  \sigma_{p+q+1}^{s+1}\circ (P\diamond Q)\circ    \sigma_{p+q+1}^{-(s+1)}\\
			=&(-1)^{(p+q)^2} P\diamond Q+\sum_{s=0}^{p+q-1}(-1)^{(p+q)s}    \sigma_{p+q+1}^{s+1}\circ (P\diamond Q)\circ    \sigma_{p+q+1}^{-(s+1)} \\
			=&(-1)^{p+q}P\diamond Q+\sum_{t=1}^{p+q} (-1)^{(p+q)(t-1)}  \sigma_{p+q+1}^{t}\circ (P\diamond Q)\circ  \sigma_{p+q+1}^{-t}\\
			=&(-1)^{p+q}\Big(P\diamond Q+ \sum_{t=1}^{p+q} (-1)^{(p+q)t}    \sigma_{p+q+1}^{t}\circ (P\diamond Q)\circ  \sigma_{p+q+1}^{-t}\Big)\\
			=&(-1)^{p+q} P\square Q.
		\end{align*}
		Hence $ P\square Q\in   \mathcal{C}_{dL}^{p+q}(V)$.
	\end{proof}
	
	\begin{example}\label{ex-dl-element}
		For $\lambda$, $\mu\in \mathcal{C}_{dL}^{1}(V)$, define the
		left extension of $\lambda$ by $    \lambda_L(a,u\otimes v):=\lambda(a,u)\otimes v$.
		Then, for all $a,b,c\in V$, we have
		\begin{align*}
			(\lambda\square \mu)(a,b,c)=\lambda_L(a,\mu(b,c))+\lambda_L^{\sigma}(b,\mu(c,a))+\lambda_L^{\sigma^2}(c,\mu(a,b)),
		\end{align*}
		where $\lambda^\sigma (a,b)= \sigma (\lambda (a,b))$.
	\end{example}
	\begin{lemma}   \label{lem-three-vertex-cyclic}
		Let $P\in\mathcal{C}_{dL}^p(V)$, $Q\in\mathcal{C}_{dL}^q(V)$ and  $R\in\mathcal{C}_{dL}^r(V)$.
		For $0\leq i\leq p+q$ and $0\leq j\leq q+r$, define
		\begin{align*}
			X_i(P,Q,R)&:=(\sigma_{p+q+1}^i(P\diamond Q)\sigma_{p+q+1}^{-i}\otimes\id^{\otimes r})\circ (\id^{\otimes(p+q)}\otimes R ),\\
			Y_j(P,Q,R)&:=(P\otimes\id^{\otimes(q+r)})\circ( \id^{\otimes p}\otimes      \sigma_{q+r+1}^{j}(Q\diamond R)\sigma_{q+r+1}^{-j}).
		\end{align*}
		Then
		\begin{align}
			X_0(P,Q,R)&=Y_0(P,Q,R),                                      \label{eq-root-1}\\
			X_i(P,Q,R)&=(-1)^{qi}\sigma_{p+q+r+1}^i (Y_{q+r+1-i}(P,Q,R)) \sigma_{p+q+r+1}^{-i}, &&1\leq i\leq q,        \label{eq-root-2}\\
			X_i(P,Q,R)&=(-1)^{p(i+q)}\sigma_{p+q+r+1}^i (X_{p+q+r+1-i}(P,R,Q))\sigma_{p+q+r+1}^{-i},        &&q+1\leq i\leq p+q,                                           \label{eq-root-3}\\
			Y_j(P,Q,R)&=(-1)^{rj}\sigma_{p+q+r+1}^{p+j} (Y_{r+1-j}(Q,P,R)) \sigma_{p+q+r+1}^{-(p+j)},       &&1\leq j\leq r.                                               \label{eq-root-4}
		\end{align}
	\end{lemma}
	\begin{proof}
		We identify
		$x_1\otimes\cdots\otimes x_{p+q+r+1}$ with $(x_1,\ldots,x_{p+q+r+1})$.
		For a map $F:V^{\otimes N}\to V^{\otimes N}$, we have $(\sigma_N^aF\sigma_N^{-a})(z_1,\ldots,z_N)
		=   \sigma_N^a F(z_{a+1},\ldots,z_N,z_1,\ldots,z_a)$.
		Moreover, if $F\in\mathcal{C}_{dL}^f(V)$, then
		$$
		F(\sigma_{f+1}^a(z_1,\ldots,z_{f+1}))=(-1)^{fa}\sigma_{f+1}^aF(z_1,\ldots,z_{f+1}).
		$$
		Direct evaluation gives
		\begin{align*}
			X_0(P,Q,R) =&( (P\otimes \id^{\otimes q})(\id^{\otimes p}\otimes Q)\otimes \id^{\otimes r}) \circ (\id^{\otimes(p+q)}\otimes R )\\
			=&(P\otimes \id^{\otimes (q+r)})(\id^{\otimes p}\otimes (Q\otimes \id^{\otimes r})(\id^{\otimes q}\otimes R) )=Y_0(P,Q,R),
		\end{align*}
		which proves  \eqref{eq-root-1}.
		
		We suppress summation signs in the tensor components below.
		For  $1\leq i\leq q$, we set
		\begin{align*}
			&R(x_{p+q+1},\ldots,x_{p+q+r+1})=\rho_0\otimes\cdots\otimes\rho_r,\qquad Q(x_{p+i+1},\ldots,x_{p+q},\rho_0,x_1,\ldots,x_i)=\eta_0\otimes\cdots\otimes\eta_q,\\
			&P(x_{i+1},\ldots,x_{p+i},\eta_0)=\pi_0\otimes\cdots\otimes\pi_p.
		\end{align*}
		Then we obtain
		\begin{align*}
			&X_i(P,Q,R)(x_1,\ldots,x_{p+q+r+1})\\
			= &(\sigma_{p+q+1}^i(P\otimes \id^{\otimes q})(\id^{\otimes p}\otimes Q) \sigma_{p+q+1}^{-i}\otimes \id^{\otimes r})(x_1,\dots, x_{p+q}, R(x_{p+q+1},\ldots,x_{p+q+r+1})) \\
			=&\sigma_{p+q+1}^i (P\otimes \id^{\otimes q}) (x_{i+1},\ldots,x_{p+i},Q( x_{p+i+1},\ldots,x_{p+q},\rho_0,x_1,\ldots,x_i))\otimes \rho_1\otimes\cdots\otimes\rho_r\\
			=&  \eta_{q-i+1}\otimes\cdots\otimes\eta_q
			\otimes\pi_0\otimes\cdots\otimes\pi_p
			\otimes\eta_1\otimes\cdots\otimes\eta_{q-i}
			\otimes\rho_1\otimes\cdots\otimes\rho_r.
		\end{align*}
		On the other hand,  since $
		Q(x_1,\ldots,x_i,x_{p+i+1},\ldots,x_{p+q},\rho_0)=(-1)^{qi}\sigma_{q+1}^i(\eta_0,\ldots,\eta_q) $, we have
		\begin{align*}
			&(\sigma_{p+q+r+1}^iY_{q+r+1-i}(P,Q,R)\sigma_{p+q+r+1}^{-i})(x_1,\ldots,x_{p+q+r+1})\\
			=&\sigma_{p+q+r+1}^i   \Big( (P\otimes \id^{\otimes(q+r)} )(\id^{\otimes p}\otimes \sigma_{q+r+1}^{q+r+1-i}(Q\otimes \id^{\otimes r} )(\id^{\otimes q}\otimes R)) \\
			&   \qquad \qquad    (x_{i+1},\dots,x_{p+i},x_1, \dots, x_i, x_{i+p+1},\dots, x_{p+q+r+1})                 \Big)\\
			=&(-1)^{qi}\sigma_{p+q+r+1}^i   \Big( (P\otimes \id^{\otimes (q+r)} )( x_{i+1},\dots,x_{p+i}, \eta_0,\dots,\eta_{q-i},\rho_1,\dots,\rho_r,\eta_{q-i+1},\dots,\eta_q  ) \Big)\\
			=&(-1)^{qi} \eta_{q-i+1}\otimes\cdots\otimes\eta_q\otimes\pi_0\otimes\cdots\otimes\pi_p\otimes\eta_1\otimes\cdots\otimes\eta_{q-i}\otimes\rho_1\otimes\cdots\otimes\rho_r.
		\end{align*}
		Thus,   \eqref{eq-root-2} holds.
		
		Let $q+1\leq i\leq p+q$ and write $i=q+a$, where $1\leq a\leq p$. Set
		\begin{align*}
			&R(x_{p+q+1},\ldots,x_{p+q+r+1})=\rho_0\otimes\cdots\otimes\rho_r,\qquad Q(x_a,\ldots,x_{q+a})=\eta_0\otimes\cdots\otimes\eta_q,\\
			&P(x_{q+a+1},\ldots,x_{p+q},\rho_0, x_1,\ldots,x_{a-1},\eta_0)=\pi_0\otimes\cdots\otimes\pi_p.
		\end{align*}
		We have
		\begin{align*}
			&X_i(P,Q,R) ( x_1,\ldots,x_{p+q+r+1})\\
			=&  \pi_{p-a+1}\otimes\cdots\otimes\pi_p
			\otimes\eta_1\otimes\cdots\otimes\eta_q \otimes\pi_0\otimes\cdots\otimes\pi_{p-a}
			\otimes\rho_1\otimes\cdots\otimes\rho_r.
		\end{align*}
		On the other hand, since $P(x_1,\ldots,x_{a-1},\eta_0,  x_{q+a+1},\ldots,x_{p+q},\rho_0) =  (-1)^{pa}  \sigma_{p+1}^a(\pi_0,\ldots,\pi_p)$, we have
		\begin{align*}
			&(\sigma_{p+q+r+1}^iX_{p+q+r+1-i}(P,R,Q)\sigma_{p+q+r+1}^{-i} ) (x_1,\ldots,x_{p+q+r+1})\\
			=&\sigma_{p+q+r+1}^iX_{p+q+r+1-i}(P,R,Q)(x_{i+1},\dots,x_{p+q+r+1},x_1,\dots,x_i)\\
			=&\sigma_{p+q+r+1}^i  \Big(   (\sigma_{p+r+1}^{p+r+1-a} (P\otimes\id^{\otimes r}) (\id^{\otimes p}\otimes R) \sigma_{p+r+1}^{-(p+r+1-a)}\otimes\id^{\otimes q} )\\
			&\qquad \qquad  (x_{q+a+1},\ldots,x_{p+q+r+1}, x_1,\ldots,x_{a-1}, \eta_0,\eta_1,\ldots,\eta_q)   \Big)\\
			=&\sigma_{p+q+r+1}^i \Big(   (\sigma_{p+r+1}^{p+r+1-a} (P\otimes\id^{\otimes r}) (\id^{\otimes p}\otimes R) \otimes\id^{\otimes q})  (x_1,\ldots,x_{a-1}, \eta_0,x_{q+a+1},\ldots,x_{p+q+r+1}, \eta_1,\ldots,\eta_q)   \Big)\\
			=&(-1)^{pa} \sigma_{p+q+r+1}^i \Bigl( \pi_0\otimes\cdots\otimes\pi_{p-a} \otimes \rho_1\otimes\cdots\otimes\rho_r  \otimes\pi_{p-a+1}\otimes\cdots\otimes\pi_p \otimes\eta_1\otimes\cdots\otimes\eta_q\Bigr)\\
			=&  (-1)^{pa} \pi_{p-a+1}\otimes\cdots\otimes\pi_p \otimes \eta_1\otimes\cdots\otimes\eta_q \otimes \pi_0\otimes\cdots\otimes\pi_{p-a}\otimes \rho_1\otimes\cdots\otimes\rho_r\\
			=&  (-1)^{p(i+q)}X_i(P,Q,R)(x_1,\ldots,x_{p+q+r+1}).
		\end{align*}
		Thus, \eqref{eq-root-3} holds.

		Finally, for $1\leq j\leq r$, we  write
		\begin{align*}
			&R(x_{p+q+j+1},\ldots,x_{p+q+r+1},x_{p+1},\ldots,x_{p+j})=\rho_0\otimes\cdots\otimes\rho_r,\qquad
			Q(x_{p+j+1},\ldots,x_{p+j+q},\rho_0)=\eta_0\otimes\cdots\otimes\eta_q,\\
			&P(x_1,\ldots,x_p,\rho_{r-j+1})=\pi_0\otimes\cdots\otimes\pi_p.
		\end{align*}
		Then
		\begin{align*}
			&Y_j(P,Q,R)(x_1,\ldots,x_{p+q+r+1})\\
			=&\pi_0\otimes\cdots\otimes\pi_p\otimes\rho_{r-j+2}\otimes\cdots\otimes\rho_r\otimes\eta_0\otimes\cdots\otimes\eta_q\otimes\rho_1\otimes\cdots\otimes\rho_{r-j}.
		\end{align*}
		On the other hand, since $      R(x_{p+1},\ldots,x_{p+j},x_{p+q+j+1},\ldots,x_{p+q+r+1}) =(-1)^{rj}\sigma_{r+1}^j(\rho_0,\ldots,\rho_r)$,
		we have
		\begin{align*}
			&(  \sigma_{p+q+r+1}^{p+j}Y_{r+1-j}(Q,P,R)\sigma_{p+q+r+1}^{-(p+j)})(x_1,\ldots,x_{p+q+r+1})\\
			=&\sigma_{p+q+r+1}^{p+j}
			\Bigl(      (Q\otimes\id^{\otimes(p+r)}) (  \id^{\otimes q}\otimes\sigma_{p+r+1}^{r+1-j}
			(P\otimes\id^{\otimes r})   (\id^{\otimes p}\otimes R)  \sigma_{p+r+1}^{-(r+1-j)})
			(x_{p+j+1},\ldots,x_{p+q+r+1},x_1,\ldots,x_{p+j})\Bigr)\\
			=&(-1)^{rj} \sigma_{p+q+r+1}^{p+j}
			\Bigl(  \eta_0\otimes\cdots\otimes\eta_q\otimes\rho_1\otimes\cdots\otimes\rho_{r-j}\otimes\pi_0\otimes\cdots\otimes\pi_p
			\otimes\rho_{r-j+2}\otimes\cdots\otimes\rho_r\Bigr)\\
			=&(-1)^{rj} \pi_0\otimes\cdots\otimes\pi_p      \otimes\rho_{r-j+2}\otimes\cdots\otimes\rho_r\otimes\eta_0\otimes\cdots\otimes\eta_q\otimes\rho_1\otimes\cdots\otimes\rho_{r-j}\\
			=&(-1)^{rj}     Y_j(P,Q,R)(x_1,\ldots,x_{p+q+r+1}).
		\end{align*}
		Thus,   \eqref{eq-root-4} holds and then the proof is complete.
	\end{proof}

	\begin{theorem}\label{th-dgla}
		The bracket $[\cdot,\cdot]_{dL}$ defined by
		\begin{equation}\label{eq-dgla}
			[P,Q]_{dL}:=P\square Q-(-1)^{pq}Q\square P,\qquad P\in  \mathcal{C}^p_{dL}(V),\quad Q\in \mathcal{C}^q_{dL}(V),
		\end{equation}
		endows $\mathcal{C}_ {dL}^\bullet(V)$ with the structure
		of a graded Lie algebra. Explicitly, for $P\in  \mathcal{C}^p_{dL}(V)$, $Q\in \mathcal{C}^q_{dL}(V)$ and  $R\in\mathcal{C}_{dL}^r(V)$, we have
		\begin{align}
			\label{eq-dgla-ss}&[P,Q]_{dL}   =-(-1)^{pq}[Q,P]_{dL},\\
			\label{eq-dgla-jcb}     &[P,[Q,R]_{dL}]_{dL}=[[P,Q]_{dL},R]_{dL}+(-1)^{pq}[Q,[P,R]_{dL}]_{dL}.
		\end{align}
	\end{theorem}
	\begin{proof}
		Let $P\in  \mathcal{C}^p_{dL}(V)$, $Q\in \mathcal{C}^q_{dL}(V)$ and  $R\in\mathcal{C}_{dL}^r(V)$.
		Then we have
		\begin{align*}
			[P,Q]_{dL}=P\square Q-(-1)^{pq}Q\square P=-(-1)^{pq}(Q\square P- (-1)^{pq}P\square Q)   =-(-1)^{pq}[Q,P]_{dL}.
		\end{align*}
		
		It remains to prove   \eqref{eq-dgla-jcb}. For any
		$T\in\operatorname{End}(V^{\otimes(p+q+r+1)})$,
		write
		$$
		\operatorname{Cyc}_{p+q+r}(T):= \sum_{s=0}^{p+q+r}(-1)^{(p+q+r)s}   \sigma_{p+q+r+1}^{s}\circ T\circ    \sigma_{p+q+r+1}^{-s}.
		$$
		By a change of indices, for every integer $i$, one has
		\begin{equation}\label{eq-cyc-conjugation}
			\operatorname{Cyc}_{p+q+r}
			\left(
			\sigma_{p+q+r+1}^{i}
			T
			\sigma_{p+q+r+1}^{-i}
			\right)
			=
			(-1)^{(p+q+r)i}
			\operatorname{Cyc}_{p+q+r}(T).
		\end{equation}

		Define the associator of $\square$ by $ \mathfrak A(P,Q,R):=(P\square Q)\square R-P\square(Q\square R)$.
		Using the definitions of $X_i(P,Q,R)$ and $Y_j(P,Q,R)$ in
		Lemma \ref{lem-three-vertex-cyclic}, a direct expansion gives
		\begin{align}
			\mathfrak A(P,Q,R)= \operatorname{Cyc}_{p+q+r} \Big(\sum_{i=0}^{p+q}(-1)^{(p+q)i}X_i(P,Q,R)-\sum_{j=0}^{q+r}
			(-1)^{(q+r)j}Y_j(P,Q,R) \Big).
			\label{eq-associator-XY}
		\end{align}
		
		Now let $1\leq i\leq q$ and put $j:=q+r+1-i$.
		By \eqref{eq-root-2} and \eqref{eq-cyc-conjugation},
		\begin{align*}
			\operatorname{Cyc}_{p+q+r}  \left(      (-1)^{(p+q)i}X_i(P,Q,R) \right)
			=&(-1)^{(q+r)i} \operatorname{Cyc}_{p+q+r}\left(    Y_j(P,Q,R)\right)\\
			=&  \operatorname{Cyc}_{p+q+r}\left((-1)^{(q+r)j}Y_j(P,Q,R) \right).
		\end{align*}
		Thus all the terms with $1\leq i\leq q$ cancel with the terms
		$r+1\leq j\leq q+r$.
		
		Consequently, by \eqref{eq-root-1}, $\mathfrak A(P,Q,R)=    \mathfrak A_X(P,Q,R)+\mathfrak A_Y(P,Q,R)$,
		where
		\begin{align}
			&\mathfrak A_X(P,Q,R):=\operatorname{Cyc}_{p+q+r}    (      \sum_{i=q+1}^{p+q}  (-1)^{(p+q)i}X_i(P,Q,R) ),
			\label{eq-AX-definition}\\
			&\mathfrak A_Y(P,Q,R):= -\operatorname{Cyc}_{p+q+r}(        \sum_{j=1}^{r}      (-1)^{(q+r)j}Y_j(P,Q,R)).
			\label{eq-AY-definition}
		\end{align}
		
		Let $q+1\leq i\leq p+q$ and put $k:=p+q+r+1-i$.
		Since $ (p+q)i+p(i+q)+(p+q+r)i\equiv qr+(p+r)k  \pmod 2$,  by \eqref{eq-root-3} and \eqref{eq-cyc-conjugation}, we have
		\begin{align*}
			\operatorname{Cyc}_{p+q+r}  \left( (-1)^{(p+q)i}X_i(P,Q,R)\right) = (-1)^{qr}       \operatorname{Cyc}_{p+q+r}\left(        (-1)^{(p+r)k}X_k(P,R,Q)\right).
		\end{align*}
		Thus, $ \mathfrak A_X(P,Q,R)=   (-1)^{qr}\mathfrak A_X(P,R,Q)$.

		Similarly, let $1\leq j\leq r$ and put $l:=r+1-j$.
		By $(q+r)j+rj+(p+q+r)(p+j)  \equiv  pq+(p+r)l   \pmod 2$, \eqref{eq-root-4} and \eqref{eq-cyc-conjugation}, we have
		\begin{align*}
			-\operatorname{Cyc}_{p+q+r}\left(       (-1)^{(q+r)j}Y_j(P,Q,R)\right)
			=       (-1)^{pq}   \left(-     \operatorname{Cyc}_{p+q+r}\left(        (-1)^{(p+r)l}Y_l(Q,P,R) \right)     \right).
		\end{align*}
		Thus, $     \mathfrak A_Y(P,Q,R)=       (-1)^{pq}\mathfrak A_Y(Q,P,R)$.
		
		We now calculate \eqref{eq-dgla-jcb}. Since
		\begin{align*}
			&(P\square Q)\square R =    \operatorname{Cyc}_{p+q+r} ( \sum_{s=0}^{p+q}(-1)^{(p+q)s}X_s(P,Q,R)),\\
			&P\square (Q\square R)=\operatorname{Cyc}_{p+q+r} ( \sum_{k=0}^{q+r}(-1)^{(q+r)k}Y_k(P,Q,R)),
		\end{align*}
		expanding the definition of
		$[\cdot,\cdot]_{dL}$ and collecting terms into associators, we get
		\begin{align*}
			&[P,[Q,R]_{dL}]_{dL}-[[P,Q]_{dL},R]_{dL}-(-1)^{pq}[Q,[P,R]_{dL}]_{dL}\\
			=&-\mathfrak{A}(P,Q,R)+(-1)^{qr}\mathfrak{A}(P,R,Q)-(-1)^{(p+q)r}\mathfrak{A}(R,P,Q)\\
			&+(-1)^{pq}\mathfrak{A}(Q,P,R)-(-1)^{p(q+r)}\mathfrak{A}(Q,R,P)+(-1)^{pq+(p+q)r}\mathfrak{A}(R,Q,P)\\
			=&-\mathfrak{A}_X(P,Q,R)-\mathfrak{A}_Y(P,Q,R)+(-1)^{qr}(\mathfrak{A}_X(P,R,Q) +\mathfrak{A}_Y(P,R,Q)  )\\
			&-(-1)^{(p+q)r}(\mathfrak{A}_X(R,P,Q)+\mathfrak{A}_Y(R,P,Q) )+(-1)^{pq}(\mathfrak{A}_X(Q,P,R) +\mathfrak{A}_Y(Q,P,R))\\
			&-(-1)^{p(q+r)}(\mathfrak{A}_X(Q,R,P)+\mathfrak{A}_Y(Q,R,P)      )+(-1)^{pq+(p+q)r}(\mathfrak{A}_X(R,Q,P)+\mathfrak{A}_Y(R,Q,P))\\
			=&0.
		\end{align*}
		
		Therefore \eqref{eq-dgla-jcb} holds  and then $(\mathcal{C}_{dL}^\bullet(V), [\cdot,\cdot]_{dL})$ is a graded Lie algebra.
	\end{proof}

	\begin{corollary}\label{cor-mc-dla}
		There is a bijective correspondence between double Lie brackets $\lc\cdot,\cdot \rc:V\otimes V\rightarrow V\otimes V$ on a vector space $V$ and Maurer-Cartan elements $\mu\in\mathcal{C}_{dL}^1(V)$. Explicitly, the correspondence is  given by
		\begin{equation*}
			\lc a,b\rc=\mu(a,b),\qquad a,b\in V,
		\end{equation*}
		and the Maurer-Cartan equation is equivalently $[\mu,\mu]_{dL}=0$.
	\end{corollary}
	\begin{proof}
		By definition, $\mu \in \mathcal{C}^1_{dL}(V)$ implies that $\sigma_2\mu\sigma_2^{-1}=-\mu$.
		Thus $\mu \in \mathcal{C}^1_{dL}(V)$ if and only if $\mu(a,b)=-\mu^{\sigma_2}(b,a)$.
		Finally, by Example \ref{ex-dl-element}, we have
		\begin{equation*}
			\frac{1}{2}[\mu,\mu]_{dL}(a,b,c) =\mu_L(a,\mu(b,c))+\mu_L(b,\mu(c,a))^\sigma+\mu_L(c,\mu(a,b))^{\sigma^2},\qquad a,b,c\in V.
		\end{equation*}
		Therefore, the  skew-symmetry of $\mu$ is equivalent to
		$\mu\in\mathcal{C}_{dL}^{1}(V)$, while the Jacobi identity  is
		equivalent to $\mu\square\mu=0$. Since
		$[\mu,\mu]_{dL}=2\mu\square\mu$, the double Lie algebra structures on
		$V$ are precisely the Maurer-Cartan elements of
		$\mathcal{C}_{dL}^{\bullet}(V)$. 
	\end{proof}
	
	Let $\mu\in\mathcal{C}_{dL}^{1}(V)$ be a double Lie bracket.  Since
	$[\mu,\mu]_{dL}=0$, the graded Jacobi identity gives
	\begin{equation*}
		[\mu,[\mu,P]_{dL}]_{dL}
		=\frac12[[\mu,\mu]_{dL},P]_{dL}=0,
		\qquad P\in\mathcal{C}_{dL}^{\bullet}(V).
	\end{equation*}
	Thus $d_\mu:=[\mu,\cdot]_{dL}$ is a differential of degree one, and  $(\mathcal{C}_{dL}^\bullet (V),[\cdot,\cdot]_{dL},d_\mu)$ is a dg Lie algebra.
	
	\begin{definition}\label{def-dL-cochain}
		Let $(V,\lc\cdot,\cdot\rc)$ be a double Lie algebra, and let
		$\mu\in\mathcal{C}_{dL}^{1}(V)$ denote its double bracket $\lc\cdot,\cdot\rc$.  The
		\textbf{double Lie algebra cochain complex} is $(\mathcal{C}_{dL}^{\bullet}(V),d_\mu )$.
		For $P\in\mathcal{C}_{dL}^{p}(V)$, the differential is explicitly
		\begin{equation*}
			d_\mu (P):=\sum_{s=0}^{p+1}(-1)^{(p+1)s}\sigma_{p+2}^s\circ(\mu\diamond P)\circ \sigma_{p+2}^{-s}
			-(-1)^p \sum_{s=0}^{p+1}(-1)^{(p+1)s}\sigma_{p+2}^s\circ(P\diamond \mu)\circ \sigma_{p+2}^{-s}.
		\end{equation*}
		Its cohomology is denoted by $H_{dL}^\bullet(V,\mu)=\bigoplus_{p\geq 0}H_{dL}^p(V,\mu)$, where $H_{dL}^p(V,\mu)$ is defined by
		$$
		H_{dL}^p(V,\mu):= \frac{\ker (d_\mu: \mathcal{C}_{dL}^p(V)\rightarrow \mathcal{C}_{dL}^{p+1} (V)) } { \mathrm{im} (d_\mu: \mathcal{C}_{dL}^{p-1}(V)\rightarrow \mathcal{C}_{dL}^{p} (V))}.
		$$
		Here and throughout, we adopt the convention
		$\mathcal{C}_{dL}^{-1}(V):=0$. Hence, when $p=0$, the denominator is zero and
		$H_{dL}^{0}(V,\mu)= \ker (  d_\mu:\mathcal{C}_{dL}^{0}(V) \longrightarrow       \mathcal{C}_{dL}^{1}(V) )$.
	\end{definition}

	\begin{remark}
		For $n\geq 1$, let $\mathrm{BR}_{dLie}^n (V)$ denote the space of cyclically skew-symmetric $n$-ary maps used in  \cite[Section 4.1.4]{FV25}, and let $\widehat{d}$ denote the differential given by formula (4.3) in  \cite{FV25}.
		Then    $ \mathcal{C}_{dL}^{p}(V)=\mathrm{BR}_{\rm dLie}^{p+1}(V), p\geq0$.
		For $P\in \mathcal{C}_{dL}^p (V)$, that formula gives
		\begin{equation*}
			\widehat{d}(P)=[P,\mu]_{dL},
			\qquad          d_\mu(P)=[\mu,P]_{dL}=(-1)^{p+1}\widehat d(P).
		\end{equation*}
		Consequently, the maps $  \varphi_p: \mathcal{C}_{\mathrm{dL}}^p(V) \rightarrow   \ \mathrm{BR}^{p+1}_{\rm dLie}(V)$ defined by    $\varphi_p(P)=  (-1)^{\frac{p(p+1)}{2}}P$,  satisfy $   \varphi_{p+1}\circ d_\mu    =   \widehat d\circ\varphi_p$.
		Thus our complex is isomorphic to the positive arity part $\bigoplus_{n\geq 1}\mathrm{BR}_{\rm dLie}^n(V)$ of $\mathrm{BR}_{\rm dLie}(V)$ in  \cite{FV25}, with its degrees reindexed by $n=p+1$. This statement concerns the positive arity part and does not identify any additional arity-zero term or the coboundaries contributed by that term.
		
	\end{remark}

	\subsection{A properadic interpretation  of the cochain complex and cohomology of double Lie algebras}\label{subsec-properadic-interpretation}
	Throughout this subsection, $V$ is a vector space, regarded as a chain complex concentrated in degree $0$. Unless explicitly stated otherwise, we use cohomological grading, so that differentials have degree $1$. We follow the terminology of  \cite{Leray20,MV09}. 
	
	Let $\ch_\bfk$ denote the category of $\mathbb{Z}$-graded chain complexes over $\bfk$, and   $\Bij$ denote the groupoid of finite sets and bijections. This gives us the equivalence of categories $\Bij\cong \Bij^{op}$ by passage to the inverse.
	\begin{definition}
		A (right) $\fraks$-module is an object of $\Func(\Bij^{op},\ch_\bfk)$, the category of contravariant functors from $\Bij$ to $\ch_\bfk$, denoted by $\smc$. An $\fraks$-bimodule is an object of the category $\Func(\Bij\times \Bij^{op},\ch_\bfk)$ which is denoted by $\sbmc$.
		
		An $\fraks$-module (resp. $\fraks$-bimodule) $\mathsf {P}$ is called \textbf{reduced}  if it satisfies $\mathsf{P}(\emptyset)=0$ (resp. $\mathsf{P}(\emptyset,S)=0$ and $\mathsf{P}(S,\emptyset)=0$ for every finite set $S$).
	\end{definition}
	
	We write $\boxtimes_c$ for the connected composition product of reduced $\fraks$-modules and $\boxtimes_c^{\rm Val}$ for the connected composition product of reduced $\fraks $-bimodules. Their units are supported in arity $1$ and biarity $(1,1)$, respectively, with value $\bfk$.
	
	\begin{definition}
		
		A \textbf{protoperad} is a unital monoid  in the monoidal category ($\smc^{red}$, $\boxtimes_c$), and  a \textbf{properad} is a unital monoid in the monoidal category $(\sbmc^{\mathrm{red}},\boxtimes_c^{\mathrm{Val}})$ of   reduced $\fraks$-bimodules.  In particular, a properad  $\mathsf{P}$  is equipped with maps
		\begin{equation*}
			\gamma_{\mathsf{P}}:\mathsf{P}\boxtimes_c^{\rm Val} \mathsf{P}\longrightarrow \mathsf{P}, \qquad \eta_{\mathsf{P}}:I_{\boxtimes^{\rm Val}}\rightarrow \mathsf{P},
		\end{equation*}
		satisfying associativity and unit   relations, where $I_{\boxtimes^{\rm Val}}=\{  I_{\boxtimes^{\rm Val} } (m,n)\} $ is an $\fraks $-bimodule with all components $I_{\boxtimes^{\rm Val}}(m,n)$ vanishing except for $I_{\boxtimes^{\rm Val}}(1,1)$ which equals $\bfk$.  Coprotoperads and coproperads are defined dually as counital comonoids in the corresponding monoidal categories.
		A properad $\mathsf{P}$ is said to be \textbf{augmented} if there exists a properad morphism $\epsilon:\mathsf{P}\rightarrow I_{\boxtimes^{\rm Val}}$. We denote by $\overline{\mathsf{P}}$ the kernel of the augmentation $\epsilon$ and call it the \textbf{augmentation ideal}.
		A coproperad $\mathsf{C}$ is said to be \textbf{coaugmented} if there is a coproperad morphism $\eta: I_{\boxtimes^{\rm Val}}\rightarrow \mathsf{C}$ such that $\epsilon\circ\eta=\id_{I_{\boxtimes^{\rm Val}}}$. The cokernel of $\eta$ is denoted by $\overline{\C}$ and $\mathsf{C}\cong I_{\boxtimes^{\rm Val}}\oplus \overline{\C}$.
	\end{definition}
	Since $\epsilon\circ \eta=\id_{I_{\boxtimes^{\rm Val}}}$, there is a canonical isomorphism 
	$\operatorname{coker}\eta \xrightarrow{\cong} \ker \epsilon$.

	\begin{example}
		For a vector space $V$, the reduced endomorphism properad of $V$ is
		$$
		\en_V(m,n):=  \begin{cases}
			\Hom(V^{\otimes n},V^{\otimes m}), & m,n\geq 1,\\
			0,&m=0 \text{ or } n=0.
		\end{cases}
		$$
		Its composition is given by composition and tensor products of linear maps.
		For $\sigma\in\fraks_r$, we also write $\sigma$ for the
		induced permutation operator on $V^{\otimes r}$, defined by
		$$
		\sigma(v_1\otimes\cdots\otimes v_r)  :=  v_{\sigma^{-1}(1)}\otimes\cdots\otimes v_{\sigma^{-1}(r)}.
		$$
		With this convention, the symmetric-group actions on $\en_V$ are
		$$
		(\sigma,\tau)\cdot F :=\sigma\circ F\circ\tau^{-1}, \qquad  (\sigma,\tau)\in\fraks_m\times\fraks_n,\quad F\in\en_V(m,n).
		$$
		
	\end{example} 
	\begin{definition}
		Let $\mathsf{P}$ be a properad. A $\mathsf{P}$-algebra
		structure on $V$ is a properad morphism $\mathsf{P}\rightarrow
		\en_V$.
	\end{definition}
	For graded vector spaces $A$ and $B$, we write $\Hom^p(A,B):=\prod_{i\in \mathbb{Z}}\Hom(A^i,B^{i+p})$. Thus  $\Hom^p(A,B)$ consists of all homogeneous linear maps of degree $p$. These maps are not required to commute with the differentials. For reduced dg $\fraks$-bimodules $\mathsf{A}$ and $\mathsf{B}$, denote
	\begin{equation}\label{eq-convolution-total-degree}
		\mathsf{Hom}^p(\mathsf{A},\mathsf B):= \bigoplus_{m,n\geq1}
		\prod_{i\in\mathbb Z}   \Hom_\bfk   (   \mathsf{A}^i(m,n),  \mathsf{B}^{i+p}(m,n) ), \qquad p\in\mathbb Z,
	\end{equation}
	and set
	$ \mathsf{Hom}^\bullet(\mathsf A,\mathsf B):=\bigoplus_{p\in\mathbb Z} \mathsf{Hom}^p(\mathsf A,\mathsf B)$.
	Thus $ \mathsf{Hom}^p(\mathsf A,\mathsf B)$ consists of biarity-preserving maps of degree $p$  with finite biarity support.
	Let
	\begin{equation}\label{eq-equivariant-convolution-total-degree}
		\mathsf{Hom}_{\fraks}^p(\mathsf A,\mathsf B):=\bigoplus_{m,n\geq1}\prod_{i\in\mathbb Z}
		\Hom_{\fraks_m\times\fraks_n} ( \mathsf A^i(m,n), \mathsf B^{i+p}(m,n)),
	\end{equation}
	and $ \mathsf{Hom}_{\fraks}^\bullet(\mathsf A,\mathsf B):=\bigoplus_{p\in\mathbb Z} \mathsf{Hom}_{\fraks}^p(\mathsf A,\mathsf B)$.
	
	
	\begin{proposition} \cite[Proposition 11, Theorem 13]{MV09}
		Let $\mathsf{C}$ be a (reduced) dg coproperad and $\mathsf{P}$ be a (reduced) dg properad.  Then $( \mathsf{Hom}^\bullet(\mathsf{C},\mathsf{P}), [\cdot,\cdot]_{\rm conv},\partial )$ is a  dg Lie algebra, where the differential $\partial$ and the bracket $ [\cdot,\cdot]_{\rm conv}$ are defined by
		\begin{align}
			&\partial f:=d_{\mathsf P}\circ f-  (-1)^p f\circ d_{\mathsf C},    \label{eq-convolution-differential}\\
			&[f,g]_{\mathrm{conv}}:=f\star g-(-1)^{pq}g\star f,    \qquad f\in \mathsf{Hom}^p(\mathsf C,\mathsf P),~g\in \mathsf{Hom}^q(\mathsf C,\mathsf P), \label{eq-convolution-Lie-bracket}
		\end{align}
		where $f\star g=\gamma_{\mathsf P}  \circ\bigl(f\boxtimes_{(1,1)}^{\mathrm{Val}}g\bigr) \circ\Delta_{(1,1)} $.
		Here $\Delta_{(1,1)}$ is the infinitesimal coproduct of  $\mathsf{C}$, and the tensor product of homogeneous maps is
		evaluated using the Koszul sign convention.
		
		Moreover,  $ \mathsf{Hom}_{\fraks}^\bullet(\mathsf {C},\mathsf {P})$ is also a dg Lie subalgebra.
	\end{proposition}

	For a reduced $\fraks$-module $V$ and  finite sets $S,~E$, the induction functor $\Ind:\smc\rightarrow \sbmc$ is defined by
	\begin{equation}
		(\Ind V)(S,E)\cong \begin{cases}
			0,&\text{if } S\ncong E,\\
			\bfk[\Aut(E)]\otimes V(E), &\text{otherwise}.
		\end{cases}
	\end{equation}
	In skeletal notation,
	$(\Ind V)(m,n)=0$ if $m\neq n$, $(\Ind V)(n,n)\cong \bfk[\fraks_n]\otimes V(n)$.
	Writing $G_n:=\fraks_n\times \fraks_n$, the induced left $G_n$-module can equivalently be written as $(\Ind V)(n,n)\cong \bfk[G_n]\otimes_{\bfk[\Delta\fraks_n]}V(n) $, where the right action on $V(n)$ is converted into a left action by inversion.
	This functor is exact, has a right adjoint which is the functor of restriction $\operatorname{Res}$, and is monoidal. Hence, it induces the functor $\Ind:\mathrm{protoperads} \rightarrow \mathrm{properads}$  \cite{Leray22}.   It also commutes with the bar, cobar and Koszul dual constructions  \cite{Leray20}. 

	Let $\dlie$ be the protoperad governing double Lie algebras and
	$\D:=\Ind (\dlie)$. By  \cite[Theorem 4.7 and Corollary
	4.8]{Leray20}, the protoperad $\dlie$ is Koszul  and the properad
	$\D$ is Koszul. Thus by  \cite[Theorem 39]{MV09}, $\D$ admits a
	quadratic model: $\Omega(\D^{\text {\textexclamdown}}
	)\xrightarrow{\sim}\D$.
	
	The generator of $\D$ has biarity $(2,2)$; its diagonal skew relation and its quadratic relation become, under a morphism
	$\rho:\D\to\operatorname{End}_V$, respectively the double
	skew-symmetry and the double Jacobi identity.  Hence $\D$-algebra
	structures on $V$ are precisely double Lie algebra structures.
	
	In the homological convention of  \cite{Leray20}, for a homologically graded $\fraks$-module $M$, let $\Sigma M$  denote its aritywise suspension, so that $(\Sigma M)_k(n)=M_{k-1}(n)$. 
	By  \cite{Leray20}, for every integer $n>0$,
	$$
	\dlie^{\text {\textexclamdown}}(n)=\Sigma^{n-1}\widetilde{\operatorname{sgn}}(\mathbb{Z}/n\mathbb{Z})\uparrow^{\fraks_n}_{\mathbb{Z}/n\mathbb{Z}}  \cong  \Sigma^{n-1}\bfk_{\chi_n}\otimes_{C_n} \bfk [\fraks_n],
	$$
	where $ C_n:=\langle\sigma_n\rangle$ and $\chi_n(\sigma_n^r):=(-1)^{(n-1)r}$.
	Thus,
	\begin{equation}\label{eq-induced-coproperad}
		\D^{\text{\textexclamdown}}(m,n)\cong \begin{cases}
			\Sigma^{n-1}(\bfk[\fraks_n\times \fraks_n]\otimes_{\bfk[\Delta C_n]}\bfk_{\chi_n}), &m=n\geq1,\\
			0,&m\neq n.
		\end{cases}
	\end{equation}
	Here $\Delta C_n=\{ (c_n,c_n)\mid c_n\in C_n\}$.
	Note that $\D^{\text{\textexclamdown}}(n,n)$ is concentrated in the  homological degree $n-1$, and next we use cohomological grading, that is, reverse the grading.
	
	Let $\xi_n$ be the element dual to the standard labelled stairway, with the normalization used after Lemma 1.32 of  \cite{LV22}.
	Under \eqref{eq-induced-coproperad}, $\xi_n$ corresponds to $1\otimes1$. From now on, $ \D^{\text{\textexclamdown}}$ denotes the same coproperad with reversed grading. Thus its component of biarity $(n,n)$ is concentrated in cohomological degree $1-n$.
	Then
	\begin{equation}\label{eq-xi-cyclic-character}
		(\sigma_n^r,\sigma_n^r)\cdot\xi_n =(-1)^{(n-1)r}\xi_n,  \qquad |\xi_n|=1-n.
	\end{equation}
	For $n=1$, $\xi_1$ is the coaugmentation element.
	
	\begin{proposition}\label{prop-def-complex}
		For $p\in\mathbb Z$ and $n\geq1$, the evaluation on $\xi_n$ induces an isomorphism
		\begin{equation*}
			\Hom_{G_n}^{ p}     (\D^{\text{\textexclamdown}}(n,n),\en(V^{\otimes n}) )
			\cong
			\begin{cases}
				\mathcal C_{dL}^{p}(V),&n=p+1,\\
				0,&n\neq p+1.
			\end{cases}
		\end{equation*}
		In particular, the left-hand side vanishes for $p<0$.
		Consequently, there is an isomorphism of graded vector spaces
		\begin{equation}\label{eq-Theta-graded-space-isomorphism}
			\Theta=\bigoplus_{p\geq0}\Theta_p:
			\mathsf{Hom}_{\fraks}^\bullet(\D^{\text{\textexclamdown}},\en_V)
			\xrightarrow{\ \cong\ }\mathcal C_{dL}^\bullet(V),
		\end{equation}
		given by
		$        \Theta_p(F):=F(\xi_{p+1}) $ for $ F\in \mathsf{Hom}_{\fraks}^p(\D^{\text{\textexclamdown}},\en_V)$, where $\en_V$ denotes the reduced endomorphism properad of $V$ defined above; in particular, $\en_V(n,n)=\en(V^{\otimes n})$.
		
	\end{proposition}
	
	\begin{proof}
		The source $\D^{\text{\textexclamdown}}(n,n)$ is concentrated in cohomological degree $1-n$, and the target $\en(V^{\otimes n})$ is concentrated in degree $0$.
		A degree $p$ map between them can therefore be nonzero only if $1-n+p=0$, that is, $n=p+1$.
		
		Fix $n\geq1$ and let $F$ be a $G_n$-equivariant map of degree $n-1$. Let $P:=F(\xi_n)$. Equivariance and
		\eqref{eq-xi-cyclic-character} imply
		\begin{align*}
			\sigma_nP\sigma_n^{-1}=(\sigma_n,\sigma_n)\cdot F(\xi_n)=F\bigl((\sigma_n,\sigma_n)\cdot\xi_n\bigr)=(-1)^{n-1}P.
		\end{align*}
		Hence $P\in\mathcal C_{dL}^{n-1}(V)$.
		Since $\xi_n$ generates $\D^{\text{\textexclamdown}}(n,n)$ as a $G_n$-module,
		every element in $\D^{\text{\textexclamdown}}(n,n)$ can be written as $ u=\sum_j \lambda_j (a_j,b_j)\cdot \xi_n$, where $\lambda_j\in \bfk$ and $(a_j,b_j)\in G_n$.
		If  $\Theta_{n-1} {F}=F(\xi_n) =0$, then for every $u=\sum_j \lambda_j (a_j,b_j)\cdot \xi_n$,
		$$
		F(u)=\sum_j \lambda_jF(  (a_j,b_j)\cdot \xi_n)=\sum_j \lambda_j(a_j,b_j)\cdot F(\xi_n)=0.
		$$
		Thus $\ker \Theta_{n-1}=0$, which means that $\Theta_{n-1}$ is injective.
		
		Given $P\in\mathcal C_{dL}^{n-1}(V)$, define
		$F_P ((a,b)\otimes1):= a\circ P\circ b^{-1} $ for  $(a,b)\in G_n$.
		For $c=\sigma_n^r\in C_n$, we have
		\begin{align*}
			F_P((ac,bc)\otimes1 )= a\circ c\circ P\circ c^{-1}\circ b^{-1}=\chi_n(c)a\circ P\circ b^{-1}=F_P((a,b)\otimes\chi_n(c) ).
		\end{align*}
		Thus $F_P$   is well defined. For $(u,v)\in G_n$,
		\begin{align*}
			F_P ((u,v)(a,b)\otimes1 )=u\circ a\circ P\circ b^{-1}\circ v^{-1}=(u,v)\cdot F_P ((a,b)\otimes1),
		\end{align*}
		which means that  $F_P$ is $G_n$-equivariant. It has degree $n-1$, and $F_P(\xi_n)=P$. Thus  $\Theta_{n-1}$ is an isomorphism.
		
		Since $\D^{\text{\textexclamdown}}$ is zero off the diagonal, each degree $p$ receives a contribution from only the biarity $(p+1,p+1)$. Taking the graded total spaces proves \eqref{eq-Theta-graded-space-isomorphism}.
	\end{proof}
	
	Let $\overline{\D^{\text{\textexclamdown}}}:=\ker(\epsilon_{\D^{\text{\textexclamdown}}}:\D^{\text{\textexclamdown}}\longrightarrow I_{\boxtimes^{\rm Val}})$.
	Since $\overline{\D^{\text{\textexclamdown}}}(1,1)=0$, $\Theta$ restricts to a graded vector space isomorphism
	$ \mathsf{Hom}_{\fraks}^\bullet (\overline{\D^{\text{\textexclamdown}}},\en_V )
	\cong\bigoplus_{p\geq1}\mathcal C_{dL}^p(V)$.
	On this positive degree part, we denote by $\star$ the convolution
	product defined using the reduced infinitesimal coproduct.
	
	\begin{lemma}\label{lem-convolution-square}
		Let $F\in   \mathsf{Hom}_{\fraks}^{p}(\overline{\D^{\text{\textexclamdown}}},\en_V)$, $G\in \mathsf{Hom}_{\fraks}^{q}(\overline{\D^{\text{\textexclamdown}}},\en_V)$ for $p,q\geq1$.
		Then $F\star G\in \mathsf{Hom}^{p+q}_\fraks(\overline{\D^{\text{\textexclamdown}}},\en_V)$ and
		\begin{equation}\label{eq-convolution-square-intertwining}
			\Theta_{p+q}(F\star G)= \Theta_p(F)\square\Theta_q(G).
		\end{equation}
	\end{lemma}
	
	\begin{proof}
		Let $P:=F(\xi_{p+1})$, $Q:=G(\xi_{q+1})$.
		By Proposition \ref{prop-def-complex}, $F$ and $G$ are supported
		only in biarities $(p+1,p+1)$ and $(q+1,q+1)$, respectively.
		Since $\overline\Delta_{(1,1)}$ and $\gamma_{\en_V}$ are degree $0$ maps, the convolution product
		$$
		F\star G    =\gamma_{\en_V}\circ    (F\boxtimes_{(1,1)}^{\mathrm{Val}}G)\circ\overline\Delta_{(1,1)} \in \mathsf{Hom}^{p+q}_\fraks(\overline{\D^{\text{\textexclamdown}}},\en_V).
		$$
		Then by Proposition \ref{prop-def-complex}, $F\star G$ is concentrated in biarity $(p+q+1,p+q+1)$. Moreover, since $F\star G$ is $G_{p+q+1}$-equivariant and $\xi_{p+q+1}$ generates $\overline{\D^{\text{\textexclamdown}} }(p+q+1,p+q+1)$ as a $G_{p+q+1}$-module, it suffices to compute $(F\star G)(\xi_{p+q+1})$.
		
		We now show that only graphs with a single internal edge can contribute to this value. Consider a graph occurring in    $\overline\Delta_{(1,1)}(\xi_{p+q+1})$ whose decoration is not annihilated by $F\boxtimes_{(1,1)}^{\mathrm{Val}}G$.
		After suppressing identity vertices, this is a connected two-vertex graph. By the support conditions on $F$ and $G$,
		its vertices have biarities $(p+1,p+1)$ and $(q+1,q+1)$.
		Let $e$ be the number of internal edges joining these vertices. Each internal edge uses one output flag and one input flag.
		Hence the external biarity of the graph is $((p+1)+(q+1)-e,\ (p+1)+(q+1)-e)$.
		Since this external biarity is $(p+q+1,p+q+1)$, we obtain $e=1$.
		Thus every potentially nonzero contribution to  $(F\star G)(\xi_{p+q+1})$ comes from a graph in which the two vertices are joined by exactly one internal edge.
		
		We  use the explicit  infinitesimal decomposition formula following  \cite[Lemma 1.32]{LV22}, that is,
		$$
		\overline{\Delta_{(1,1)}}( \xi_{p+q+1})= \sum_{2\leq k\leq p+q,\sigma_{p+q+1}^r\in C_{p+q+1}}(-1)^{(k-1)(p+q+1-k)}\chi_{p+q+1}(\sigma_{p+q+1}^r)  \sigma_{p+q+1}^r\circ [\xi_k\otimes \xi_{p+q+2-k} ]_{k,1} \circ \sigma_{p+q+1}^{-r} ,
		$$
		where $[\xi_k\otimes \xi_{p+q+2-k} ]_{k,1}$ denotes the class in the infinitesimal composition product obtained by  joining the first output of the second decoration to the $k$-th input of the first decoration. The first decoration occupies the positions $1,\ldots,k$, and the second occupies the positions  $k,\ldots,p+q+1$, with the remaining external inputs and outputs in their natural order. We order the first decoration before the second.
		For this term, the corresponding elementary composition in  $\en_V$ is
		$\gamma_{\en_V} [P\otimes Q]_{p+1,1}=(P\otimes\id^{\otimes q})\circ(\id^{\otimes p}\otimes Q)=P\diamond Q$.
		
		Therefore, by the support conditions on $F$ and $G$, every summand with $k\neq p+1$ is annihilated by
		$F\boxtimes_{(1,1)}^{\mathrm{Val}}G$. Hence we have
		\begin{align*}
			\Theta_{p+q}(F\star G)=&(F\star G)(\xi_{p+q+1})=\gamma_{\en_V}  \circ  (F\boxtimes_{(1,1)}^{\rm Val} G)(\overline{\Delta_{(1,1)} }( \xi_{p+q+1}) )\\
			=&   \sum_{\sigma_{p+q+1}^r\in C_{p+q+1}}(-1)^{pq}\chi_{p+q+1}(\sigma_{p+q+1}^r)  \sigma_{p+q+1}^r ((-1)^{|G| | \xi_{p+1}|}P\diamond Q)\sigma_{p+q+1}^{-r}   \\
			&=  \sum_{r=0}^{p+q}(-1)^{(p+q)r}\sigma_{p+q+1}^r\circ(P\diamond Q)\circ\sigma_{p+q+1}^{-r}\\
			&=P\square Q.
		\end{align*}
		This proves \eqref{eq-convolution-square-intertwining}.
	\end{proof}
	
	We now consider the degree $0$ term. Since $\D^{\text{\textexclamdown}}(1,1)=\bfk$,
	$ \Hom_{\fraks}(I_{\boxtimes^{\rm Val}},\en_V)\cong \en(V)$, which means that
	$ \mathsf{Hom}_\fraks^\bullet( \D^{\text{\textexclamdown}},\en_V )=
	\en(V)\oplus \mathsf{Hom}_{\fraks}^\bullet(\overline{\D^{\text{\textexclamdown}}},\en_V )$
	as graded vector spaces.
	
	For $h\in\en(V)$, let $D_h^{(n)}:=\sum_{i=1}^{n}\id^{\otimes(i-1)}\otimes h\otimes\id^{\otimes(n-i)}$.
	Retain the convolution bracket in positive degrees and extend it
	to the full graded space by
	\begin{align}
		&[h,k]_{\mathrm{aug}} :=hk-kh,\qquad h,k\in\en(V),
		\label{eq-augmented-degree-zero-bracket}\\
		&[h,F]_{\mathrm{aug}}(c):=D_h^{(n)}F(c)-F(c)D_h^{(n)},
		\qquad
		c\in\D^{\text{\textexclamdown}}(n,n),
		\label{eq-augmented-mixed-bracket}
	\end{align}
	for $F$ of positive degree, and by graded skew-symmetry. Thus $[\cdot,\cdot]_{\mathrm{aug}}$ denotes the convolution bracket augmented by the change-of-basis action; its underlying graded space is still
	$ \mathsf{Hom}_{\fraks}^\bullet
	(\D^{\text{\textexclamdown}},\en_V)$.
	
	\begin{theorem}\label{th-properadic-deformation-complex}
		The isomorphism of graded vector spaces in Proposition \ref{prop-def-complex} is an
		isomorphism of graded Lie algebras
		\begin{equation}\label{eq-Theta-graded-Lie-isomorphism}
			\Theta:(        \mathsf{Hom}_{\fraks}^\bullet
			(\D^{\text{\textexclamdown}},\en_V),[\cdot,\cdot]_{\mathrm{aug}})
			\xrightarrow{\ \cong \ }( \mathcal {C}_{dL}^\bullet(V),
			[\cdot,\cdot]_{dL}).
		\end{equation}
		
		Let $\mu$ be a double Lie bracket on $V$, let
		$\rho_\mu:\D\to\en_V$ be the corresponding structure morphism,
		and let
		$\kappa:\D^{\text{\textexclamdown}}\to\D$
		be the canonical Koszul twisting morphism.
		Then $\Theta$ induces an isomorphism of dg Lie
		algebras
		\begin{equation}\label{eq-Theta-dg-Lie-isomorphism}
			\Theta_\mu: ( \mathsf{Hom}_{\fraks}^\bullet   (\D^{\text{\textexclamdown}},\en_V),[\cdot,\cdot]_{\mathrm{aug} } ,[\alpha_\mu,\cdot]_{\mathrm{aug}})
			\xrightarrow{\ \cong\ }
			(\mathcal C_{dL}^\bullet(V),    [\cdot,\cdot]_{dL},d_\mu),
		\end{equation}
		where $\alpha_\mu:=\rho_\mu\circ\kappa$ and $d_\mu=[\mu,\cdot]_{dL}$.
	\end{theorem}
	
	\begin{proof}
		We first verify that the degree $0$ formulas define a graded Lie extension of the reduced convolution algebra.
		
		For $h,k\in\en(V)$, ordinary commutators satisfy
		$   [D_h^{(n)},D_k^{(n)}]=D_{hk-kh}^{(n)}$.
		Moreover, for
		$P\in\mathcal C_{dL}^p(V)$ and
		$Q\in\mathcal C_{dL}^q(V)$ with $p,q\geq1$, the identity
		$[A,BC]=[A,B]C+B[A,C]$ gives
		\begin{align*}
			D_h^{(p+q+1)}(P\diamond Q)-(P\diamond Q)D_h^{(p+q+1)}=
			(D_h^{(p+1)}P-PD_h^{(p+1)})\diamond Q+P\diamond(D_h^{(q+1)}Q-QD_h^{(q+1)}).
		\end{align*}
		Since $D_h^{(p+q+1)}$ commutes with
		$\sigma_{p+q+1}$, summing the cyclic conjugates yields
		\begin{align*}
			D_h^{(p+q+1)}(P\square Q)-(P\square Q)D_h^{(p+q+1)}=(D_h^{(p+1)}P-PD_h^{(p+1)})\square Q+P\square(D_h^{(q+1)}Q-QD_h^{(q+1)}).
		\end{align*}
		By Lemma~\ref{lem-convolution-square}, the action in
		\eqref{eq-augmented-mixed-bracket} is therefore an action by
		derivations of the reduced convolution product, and hence of its
		graded commutator. Thus
		\eqref{eq-augmented-degree-zero-bracket} and
		\eqref{eq-augmented-mixed-bracket} define the corresponding
		semidirect-product graded Lie bracket.
		
		Now let $F$ and $G$ have positive degrees $p$ and $q$, and set
		$P:=\Theta_p(F)$ and $Q:=\Theta_q(G)$.
		Lemma~\ref{lem-convolution-square} gives
		\begin{align*}
			\Theta_{p+q}([F,G]_{\mathrm{aug}})
			&=
			\Theta_{p+q}(F\star G)
			-
			(-1)^{pq}\Theta_{p+q}(G\star F)\\
			&=
			P\square Q-(-1)^{pq}Q\square P\\
			&=
			[P,Q]_{dL}.
		\end{align*}
		
		For $h,k\in\en(V)$, the assertion follows from $\Theta_0=\id_{\en(V)}$ and the fact that both brackets are the
		ordinary commutator.
		
		For $h\in\en(V)$ and $P=\Theta_p(F)\in\mathcal C_{dL}^p(V)$, $p\geq1$, cyclicity gives
		\begin{align*}
			h\square P&=\sum_{r=0}^{p}(-1)^{pr}\sigma_{p+1}^r(h\otimes\id^{\otimes p})P\sigma_{p+1}^{-r}=(\sum_{r=0}^{p}\sigma_{p+1}^r(h\otimes\id^{\otimes p})\sigma_{p+1}^{-r})P=D_h^{(p+1)}P,
		\end{align*}
		and similarly,
		\begin{align*}
			P\square h&=\sum_{r=0}^{p}(-1)^{pr}\sigma_{p+1}^r   P(\id^{\otimes p}\otimes h)\sigma_{p+1}^{-r}=   P(\sum_{r=0}^{p}\sigma_{p+1}^r(\id^{\otimes p}\otimes h)\sigma_{p+1}^{-r})= PD_h^{(p+1)}.
		\end{align*}
		Hence
		\begin{align*}
			\Theta_p([h,F]_{\mathrm{aug}})=D_h^{(p+1)}P-PD_h^{(p+1)}=h\square P-P\square h=[h,P]_{dL}.
		\end{align*}
		The reversed mixed bracket follows from graded skew-symmetry.
		Thus $\Theta$ preserves brackets in every degree, and hence is an isomorphism of  graded Lie algebras.
		
		The canonical Koszul twisting morphism is supported in weight $1$. With the chosen normalization of the binary generator, $ \alpha_\mu(\xi_2)=\mu$, $\alpha_\mu(\xi_n)=0$ if $n\neq2$.
		Therefore $\Theta_1(\alpha_\mu)=\mu$. Since
		$[\mu,\mu]_{dL}=0$, bracket compatibility and the injectivity of
		$\Theta$ give $ [\alpha_\mu,\alpha_\mu]_{\mathrm{aug}}=0$.
		Both $\D^{\text{\textexclamdown}}$ and $\en_V$ have zero
		internal differential. Their twisted convolution differential is $[\alpha_\mu,\cdot]_{\mathrm{aug}}$, and
		\begin{align*}
			\Theta ([\alpha_\mu,F]_{\mathrm{aug}})=[\Theta(\alpha_\mu),\Theta(F)]_{dL}=[\mu,\Theta(F)]_{dL}=
			d_\mu\Theta(F)
		\end{align*}
		for every homogeneous $F$.
		This proves \eqref{eq-Theta-dg-Lie-isomorphism}.
	\end{proof}
	\begin{remark}\label{rem-prop-complex}
		Theorem \ref{th-properadic-deformation-complex} provides a properadic interpretation of the cochain complex of a double Lie  algebra. 
		Let $\varepsilon:\D_\infty:=\Omega( \D^{\text{ \textexclamdown }   }   )\rightarrow \D$ be the cofibrant resolution of $\D=\Ind(\dlie)$  \cite[Proposition 5.7]{Leray20}. For a double Lie bracket $\mu$, let $ \rho_\mu: \D\rightarrow \en_V$ be the structure morphism and $\zeta:=\rho_\mu\circ \varepsilon$. Thus we have the commutative diagram
		\begin{equation*}
			\resizebox{0.25\textwidth}{!}{$
				\xymatrix@C=1.4cm@R=1.1cm{
					\Omega(\D^{\text{ \textexclamdown }   })
					\ar[r]^{\varepsilon}            \ar[dr]_{\zeta}
					& \D \ar[d]^{\rho_\mu}\\
					& \en_V.}       $}
		\end{equation*}
		The twisting morphism corresponding to $\zeta$ is $ \alpha_\mu=\rho_\mu\circ \kappa$, where $\kappa: \D^{\text{ \textexclamdown }   }\rightarrow \D$ is the canonical Koszul twisting morphism.
		By  \cite{MV092},  the reduced deformation complex associated with
		this resolution is isomorphic to $(
		\mathsf{Hom}_\fraks^\bullet(\overline{\D^{\text{\textexclamdown}}},\en_V
		), [\alpha_\mu,\cdot]_{\rm conv})$. As explained in the remark
		following  \cite[Theorem 12]{MV092}, this complex admits an
		augmentation by $\Hom_\fraks(I_{\boxtimes^{\rm Val}},\en_V)\cong
		\en(V)$, which is concentrated in degree $0$ in the present
		situation. Thus, by  \cite{MV092}, $ (
		\mathsf{Hom}_{\fraks}^\bullet(\D^{\text{\textexclamdown}},\en_V )
		,d_{\alpha_\mu} ) $ is the \textbf{deformation complex} of
		$\zeta:\D_\infty\rightarrow \en_V$. In this paper, we also call it
		the deformation complex of  the double Lie algebra $(V,\mu)$.
		
		Using the degree $0$ action defined above and then twisting by $\alpha_\mu$, the deformation complex of $(V,\mu)$ has a  convolution dg Lie algebra structure $( \mathsf{Hom}_{\fraks}^\bullet
		(\D^{\text{\textexclamdown}},\en_V),[\cdot,\cdot]_{\mathrm{aug}}, d_{\alpha_\mu} )$,   where $d_{\alpha_\mu}= [\alpha_\mu,\cdot]_{\rm aug}$.
		By Theorem \ref{th-properadic-deformation-complex}, this dg Lie algebra is isomorphic to
		$(\mathcal{C}_{dL}^\bullet(V),[\cdot,\cdot]_{dL},d_\mu)$. Thus the double Lie algebra cochain complex  is identified
		with the  deformation complex of $(V,\mu)$ associated with the Koszul resolution of $\D$.
		
		The cohomology of the deformation complex of $(V,\mu)$ is denoted by
		$$
		H^\bullet( \mathsf{Hom}_\fraks^\bullet(\D^{\text{\textexclamdown}},\en_V ), d_{\alpha_\mu})=\bigoplus_{p\geq 0} H^p( \mathsf{Hom}_\fraks^\bullet(\D^{\text{\textexclamdown}},\en_V ), d_{\alpha_\mu} ),
		$$
		where
		$H^p( \mathsf{Hom}_\fraks^\bullet(\D^{\text{\textexclamdown}},\en_V ), d_{\alpha_\mu} )$ is defined by
		$$
		H^p( \mathsf{Hom}_\fraks^\bullet(\D^{\text{\textexclamdown}},\en_V ), d_{\alpha_\mu} ):= \frac{\ker (d_{\alpha_\mu}: \mathsf{Hom}_{\fraks}^p(\D^{\text{\textexclamdown}},\en_V )\rightarrow \mathsf{Hom}_{\fraks}^{p+1}(\D^{\text{\textexclamdown}},\en_V )  ) } { \mathrm{im} (d_{\alpha_\mu}: \mathsf{Hom}_{\fraks}^{p-1}(\D^{\text{\textexclamdown}},\en_V )\rightarrow \mathsf{Hom}_{\fraks}^p(\D^{\text{\textexclamdown}},\en_V ))},\quad p\geq 1,
		$$
		and $H^0( \mathsf{Hom}_\fraks^\bullet(\D^{\text{\textexclamdown}},\en_V ), d_{\alpha_\mu} ) = \ker (d_{\alpha_\mu}: \mathsf{Hom}_{\fraks}^0(\D^{\text{\textexclamdown}},\en_V )\rightarrow \mathsf{Hom}_{\fraks}^{ 1}(\D^{\text{\textexclamdown}},\en_V )  ) $.
	\end{remark}

	\begin{corollary}\label{cor-dcc-dlc}
		Let $(V,\mu)$ be a double Lie algebra, and let $\alpha_\mu=\rho_\mu\circ \kappa$ be the corresponding twisting morphism.
		For every $p\geq 0$,  $\Theta_p $ induces an isomorphism
		$$
		H^p ( \mathsf{Hom}_\fraks^\bullet( \D^{\text {\textexclamdown} }, \en_V ), d_{\alpha_\mu} ) \xrightarrow{\ \cong \ } H_{dL}^p(V,\mu).
		$$
	\end{corollary}
	\begin{proof}
		It is straightforward by $d_\mu \circ \Theta_p= \Theta_{p+1}\circ d_{\alpha_\mu}$ from Theorem \ref{th-properadic-deformation-complex}.
	\end{proof}

	\section{Double Lie algebra cohomology, skew-symmetric Rota-Baxter operator cohomology and cyclic cohomology}\label{sec-dgla-rb}
	We first construct the graded Lie algebra whose Maurer-Cartan elements are precisely  skew-symmetric Rota-Baxter operators on a symmetric Frobenius algebra.
	Then we  lift the correspondence between finite-dimensional
	double Lie algebra structures on a vector space $V$ and skew-symmetric
	Rota-Baxter operators  $\mr$ on the matrix algebra $A=\en(V)$
	from the level of algebraic structures to the level of their governing differential graded Lie algebras.
	Finally, we show that the cohomology of a double Lie algebra can be described as the cyclic cohomology of the associated descendent associative algebra $A_\mr$.

	\begin{definition} \cite{Bax}
		Let $(A,\cdot)$ be an associative algebra. A linear map $\mr: A\rightarrow A$ is called a \textbf{Rota-Baxter operator} if the following identity holds:
		\begin{equation}\label{eq-rb}
			\mr(x)\cdot \mr(y)=\mr(\mr(x)\cdot y+x\cdot \mr(y)),\qquad x,y\in A.
		\end{equation}
		A \textbf{Rota-Baxter algebra} is a triple $(A,\cdot,\mr)$ consisting of an algebra $(A,\cdot)$ and a Rota-Baxter operator $\mr$.
	\end{definition}
	Then $A$ carries another associative algebra structure with the product
	$$
	x\cdot_\mr y= \mr(x)\cdot  y+x\cdot  \mr(y),\qquad x,y\in A.
	$$
	The associative algebra $(A, \cdot_\mr)$ is called the \textbf{descendent associative algebra}, and is denoted by $A_\mr$. It is obvious by \eqref{eq-rb} that $\mr$ is an associative algebra homomorphism from $A_\mr$ to $A$:
	$$
	\mr (x\cdot_\mr y)=\mr(x) \cdot  \mr(y),\qquad x,y\in A.
	$$
	Moreover, $\mr$ is still a Rota-Baxter operator on the descendent associative algebra $A_\mr$.
	
	Recall that a  bilinear form $\langle\cdot,\cdot \rangle : A\otimes A\rightarrow \bfk$ on an associative algebra $A$ is  \textbf{invariant} if $\langle xy,z\rangle =\langle x,yz\rangle $ for $x,y,z \in A$.
	A \textbf{Frobenius algebra} $(A,\langle\cdot,\cdot\rangle ) $ is an associative algebra $A$ with a non-degenerate invariant bilinear form $\langle\cdot,\cdot\rangle$.  A Frobenius algebra $(A,\langle\cdot,\cdot\rangle)$ is symmetric if $\langle\cdot,\cdot\rangle$ is symmetric.
	
	Let $V$ be a finite-dimensional vector space over $\bfk$ and   $A:=\en(V)$ equipped with a non-degenerate symmetric invariant bilinear form   $\langle x,y\rangle:=\operatorname{tr}(xy)$ for $ x,y\in A$.
	Fix a linear isomorphism $\iota: \en(V)\rightarrow (\en(V))^*$ given by  $ \iota (x) (y)=\langle x,y \rangle$. For a linear map $\mr:\en(V)\rightarrow \en(V)$, its adjoint $\mr^T$ is determined
	by
	\begin{equation*}
		\langle\mr(x),y\rangle=\langle x,\mr^T(y)\rangle,
		\qquad x,y\in \en(V).
	\end{equation*}
	We call $\mr$ \textbf{skew-symmetric} if $\mr^T=-\mr$.

	Let $V$ be a  vector space, and let $\{e_i\}_{i=1}^{n }$ and $\{e_i^T\}_{i=1}^{n }$ be trace-dual bases of $\en(V)$.
	Recall  \cite{GK18}  the following explicit expression for a double Lie bracket in terms of operators:
	\begin{equation}\label{eq-dl-rb}
		\lc a,b\rc= \sum_{i=1}^n e_i(a)\otimes \mr(e_i^T)(b)=\sum_{i=1}^n \mr^T(e_i)(a)\otimes e_i^T(b), \qquad a,b\in V.
	\end{equation}
	Every linear map $V\otimes V\rightarrow V\otimes V$ is
	obtained uniquely in this way.  The expression is independent of the chosen trace-dual bases,
	since $\sum_i e_i\otimes e_i^T$ is the canonical Casimir tensor of the trace pairing.
	\begin{theorem} \cite{GK18}\label{th-dl-rb}
		Let $V$ be a finite-dimensional vector space with a double  bracket $\{\!\{ \cdot,\cdot\}\!\}$ determined by an operator $\mr:\en(V)\rightarrow \en(V)$ by \eqref{eq-dl-rb}. Then $V$ is a double Lie algebra if and only if $\mr$ is a skew-symmetric Rota-Baxter operator on $\en(V)$.
	\end{theorem}

	\subsection{The correspondence between controlling dg Lie algebras of double Lie algebras and skew-symmetric Rota--Baxter operators}
	In  \cite{Das20}, Das constructed a graded Lie algebra whose Maurer-Cartan elements are $\mathcal{O}$-operators on an associative algebra with respect to its bimodule.
	We apply this   construction to the adjoint bimodule of a symmetric Frobenius algebra $(A,\langle\cdot,\cdot\rangle)$ and then identify its cyclic  part with the graded Lie algebra $(\mathcal{C}_{dL}^{\bullet}(V),[\cdot,\cdot]_{dL})$ when $A=\en(V)$.
	
	For $p\geq0$, with $\Hom(A^{\otimes0},A)=A$, define the
	\textbf{cyclic Rota-Baxter cochain space} by
	$$
	\mc^p_{rb}(A):= \left\{F\in\Hom (A^{\otimes p},A)\ \middle|\ \langle F(x_1,\ldots,x_p),x_0\rangle= (-1)^p\langle F(x_0,\ldots,x_{p-1}),x_p\rangle \right\}
	$$
	where $x_0,\ldots,x_p\in A$.
	Set $\mc_{rb}^{\bullet}(A):=\bigoplus_{p\geq0}\mc_{rb}^{p}(A)$ and assign
	degree $p$ to an element of $\mc_{rb}^{p}(A)$.  In particular,
	$\mc_{rb}^{0}(A)=A$.
	
	We next recall  \cite{Das20} the graded Lie bracket controlling Rota-Baxter operators.
	For $F\in\Hom(A^{\otimes p},A)$ and $G\in\Hom(A^{\otimes q},A)$ with
	$p,q\geq1$, define $F\circledcirc G\in\Hom(A^{\otimes(p+q)},A)$ by
	\begin{align*}
		(F\circledcirc G)(x_1,\ldots,x_{p+q}):=& \sum_{i=1}^{p}(-1)^{(i-1)q}
		F\bigl(x_1,\ldots,x_{i-1},  G(x_i,\ldots,x_{i+q-1})x_{i+q}, x_{i+q+1},\ldots,x_{p+q}\bigr)                                     \\
		&-\sum_{i=1}^{p}(-1)^{iq}F\bigl(x_1,\ldots,x_{i-1},  x_iG(x_{i+1},\ldots,x_{i+q}),x_{i+q+1},\ldots,x_{p+q}\bigr)                                     \\
		&+(-1)^{pq}F(x_1,\ldots,x_p)G(x_{p+1},\ldots,x_{p+q}).
	\end{align*}
	Set $      [F,G]_{rb}:=F\circledcirc G-(-1)^{pq}G\circledcirc F$.
	For $a\in A$ and  $F\in\Hom(A^{\otimes p},A)$, $p\geq1$, one has
	\begin{align}
		[F,a]_{rb}(x_1,\ldots,x_p):=&{\sum}_{i=1}^{p}F(x_1,\ldots,x_{i-1},ax_i-x_ia,x_{i+1},\ldots,x_p)                                               \notag\\
		& +F(x_1,\ldots,x_p)a-aF(x_1,\ldots,x_p).                 \label{eq-rb-degree-zero-2}
	\end{align}
	For $a,b\in A$, we define $[a,b]_{rb}=ab-ba$.
	
	\begin{theorem} \cite{Das20}\label{th-dgla-rb}
		The graded vector space $\bigoplus_{n\geq 0}\Hom(A^{\otimes n}, A )$ together with the above defined bracket $[\cdot,\cdot]_{rb}$ forms a graded Lie algebra. A linear map $\mr: A\rightarrow A$ is a Rota-Baxter operator on $A$ if and only if $\mr$ is a Maurer-Cartan element in $(\bigoplus_{n\geq 0}\Hom(A^{\otimes n}, A ),[\cdot,\cdot]_{rb} )$, that is, $\mr$ satisfies $[\mr,\mr]_{rb}=0$.
	\end{theorem}
	\begin{theorem}\label{th-dgla-ssrb}
		Let $(A,\langle\cdot,\cdot\rangle)$ be a symmetric Frobenius algebra. For all $p,q\geq0$, one has
		$[\mc_{rb}^{p}(A),\mc_{rb}^{q}(A)]_{rb}\subseteq\mc_{rb}^{p+q}(A)$.
		Consequently,
		$(\mc_{rb}^{\bullet}(A),[\cdot,\cdot]_{rb})$    is a graded Lie subalgebra of   $(\bigoplus_{n\geq 0}\Hom(A^{\otimes n}, A ),[\cdot,\cdot]_{rb} )$.
	\end{theorem}
	\begin{proof}
		For $F\in\Hom(A^{\otimes p},A)$, define the associated
		$(p+1)$-linear form
		$$
		\widehat F(x_0,\ldots,x_p):=\langle F(x_1,\ldots,x_p),x_0\rangle,\qquad x_0,\dots ,x_p\in A.
		$$
		By definition, $    F\in\mc_{rb}^{p}(A)$    if and only if $    \widehat F\circ\sigma_{p+1} =   (-1)^p\widehat F$.
		
		We first consider the case where one of the two cochains has degree
		zero. For $F\in\mc_{rb}^{p}(A)$, $p\geq1$, and   $a\in A$, we have
		\begin{align*}
			\langle    F(x_1,\ldots,x_p)a-aF(x_1,\ldots,x_p),x_0\rangle=\langle F(x_1,\ldots,x_p),ax_0-x_0a \rangle.
		\end{align*}
		It follows that
		\begin{align}
			\widehat{[F,a]_{rb}}(x_0,\ldots,x_p)=\sum_{i=0}^{p}\widehat F(x_0,\ldots,x_{i-1},[a,x_i],x_{i+1},\ldots,x_p),
			\label{eq-degree-zero-cyclic-action}
		\end{align}
		where $[a,x_i]=ax_i-x_i a$. Thus $\widehat {[F,a]_{rb}}\circ\sigma_{p+1}=(-1)^p\widehat {[F,a]_{rb}}$ since $   \widehat F\circ\sigma_{p+1} =   (-1)^p\widehat F$, and then we have $[F,a]_{rb}\in\mc_{rb}^{p}(A)$.
		Since   $[a,F]_{rb}=-[F,a]_{rb}$, the same conclusion holds for $[a,F]_{rb}$. The case $p=q=0$ follows from $[a,b]_{rb}=ab-ba\in A=\mc_{rb}^0(A)$.
		
		Suppose now that $p,q\geq1$.    For $F\in\mc_{rb}^{p}(A)$ and   $G\in\mc_{rb}^{q}(A)$, we claim that
		\begin{equation}\label{eq-RB-bracket-cyclic-symmetrization}
			\widehat{[F,G]_{rb}}(x_0,\ldots,x_{p+q}) =\sum_{r=0}^{p+q}(-1)^{(p+q)r}\theta_{F,G}(\sigma_{p+q+1}^{-r}(x_0,\ldots,x_{p+q})).
		\end{equation}
		where
		\begin{align*}
			\theta_{F,G}(x_0,\ldots,x_{p+q}):=&(-1)^{pq}\langle F(x_1,\ldots,x_p)G(x_{p+1},\ldots,x_{p+q}),x_0\rangle\\
			&-\langle G(x_1,\ldots,x_q)F(x_{q+1},\ldots,x_{p+q}),x_0\rangle.
		\end{align*}
		
		In the  following calculation, the subscripts of the $x_i$ are read modulo $p+q+1$ and in cyclic order.
		For $0\leq r\leq p+q$, set
		\begin{align*}
			A_r:=&(-1)^{(p+q)r+pq}\langle   F(x_{r+1},\ldots,x_{r+p}) G(x_{r+p+1},\ldots,x_{r+p+q}),x_r\rangle,\\
			B_r:=&-(-1)^{(p+q)r} \langle G(x_{r+1},\ldots,x_{r+q})F(x_{r+q+1},\ldots,x_{r+p+q}),x_r\rangle.
		\end{align*}
		Then $  A_r+B_r =(-1)^{(p+q)r}  \theta_{F,G}\bigl( \sigma_{p+q+1}^{-r}(x_0,\ldots,x_{p+q}))$.
		
		The two product terms in $[F,G]_{rb}$ give $A_0+B_0$.
		
		We next identify the insertion terms.  Let $1\leq i\leq p$.
		Using the cyclicity of $\widehat F$ $i$ times,  we obtain
		\begin{align*}
			&(-1)^{(i-1)q}\langle F (x_1,\dots,x_{i-1}, G(x_i,\dots,x_{i+q-1})x_{i+q},x_{i+q+1},\dots,x_{p+q}),x_0\rangle\\
			=&(-1)^{(i-1)q+pi}\langle F(x_{i+q+1},\dots,x_{i-1})G(x_i,\dots,x_{i+q-1}),x_{i+q}\rangle=A_{q+i}.
		\end{align*}
		Indeed, $   (i-1)q+pi \equiv (p+q)(q+i)+pq\pmod 2$.
		Similarly,
		\begin{align*}
			&-(-1)^{iq}\langle  F (x_1,\ldots,x_{i-1},x_iG(x_{i+1},\ldots,x_{i+q}), x_{i+q+1},\ldots,x_{p+q}),x_0\rangle\\
			=&-(-1)^{iq+pi}\langle  G(x_{i+1},\ldots,x_{i+q})F(x_{i+q+1},\ldots,x_{i-1}),x_i\rangle=B_i.
		\end{align*}
		
		Let now $1\leq j\leq q$.  Applying the same argument to the
		insertion terms coming from $G\circledcirc F$, we get
		\begin{align*}
			&-(-1)^{pq+(j-1)p}\langle G (x_1,\ldots,x_{j-1}, F(x_j,\ldots,x_{j+p-1})x_{j+p}, x_{j+p+1},\ldots,x_{p+q}),x_0\rangle\\
			=&-(-1)^{pq+(j-1)p+qj}\langle G(x_{j+p+1},\ldots,x_{j-1}) F(x_j,\ldots,x_{j+p-1}),x_{j+p}\rangle=B_{p+j},
		\end{align*}
		where $ pq+(j-1)p+qj\equiv (p+q)(p+j)\pmod2$,
		and
		\begin{align*}
			&(-1)^{pq+jp}\langle G (x_1,\ldots,x_{j-1},x_jF(x_{j+1},\ldots,x_{j+p}),x_{j+p+1},\ldots,x_{p+q}),x_0\rangle\\
			=&(-1)^{pq+jp+qj}\langle        F(x_{j+1},\ldots,x_{j+p})G(x_{j+p+1},\ldots,x_{j-1}),x_j\rangle=A_j.
		\end{align*}
		
		The terms $A_j$ with $1\leq j\leq q$ and $A_{q+i}$ with
		$1\leq i\leq p$ exhaust $A_1,\ldots,A_{p+q}$.  Likewise, the terms
		$B_i$ with $1\leq i\leq p$ and $B_{p+j}$ with
		$1\leq j\leq q$ exhaust $B_1,\ldots,B_{p+q}$.  The preceding
		identities, together with the two product terms $A_0+B_0$, therefore
		prove \eqref{eq-RB-bracket-cyclic-symmetrization}.
		
		It remains to verify cyclicity.  Set $\varepsilon:=(-1)^{p+q}$.  By
		\eqref{eq-RB-bracket-cyclic-symmetrization},
		\begin{align*}
			\widehat{[F,G]_{rb}}\circ\sigma_{p+q+1}^{-1}=\sum_{r=0}^{p+q}\varepsilon^r  \theta_{F,G}\circ\sigma_{p+q+1}^{-(r+1)}
			=\varepsilon\sum_{s=1}^{p+q+1}\varepsilon^s \theta_{F,G}\circ\sigma_{p+q+1}^{-s}=\varepsilon \widehat{[F,G]_{rb}}.
		\end{align*}
		In the last equality, we used
		$\sigma_{p+q+1}^{-(p+q+1)}=\id$ and
		$\varepsilon^{p+q+1}=(-1)^{(p+q)(p+q+1)}=1$.  Thus,
		$$
		\widehat{[F,G]_{rb}}\circ\sigma_{p+q+1}^{-1}
		=(-1)^{p+q}\widehat{[F,G]_{rb}},
		$$
		which means that $[F,G]_{rb}\in\mc_{rb}^{p+q}(A)$.
		
		Hence, $\mc_{rb}^{\bullet}(A)$ is closed under
		$[\cdot,\cdot]_{rb}$ and then   $(\mc_{rb}^{\bullet}(A),[\cdot,\cdot]_{rb})$ is a graded Lie subalgebra of  $(\bigoplus_{n\geq 0}\Hom(A^{\otimes n}, A ),[\cdot,\cdot]_{rb} )$.
	\end{proof}
	
	\begin{corollary}\label{cor-mc-ssrb}
		A linear map $\mr:A\rightarrow A$ is a   skew-symmetric Rota-Baxter operator  on a symmetric Frobenius algebra $(A,\langle\cdot,\cdot\rangle)$ if and only if $\mr $ is a Maurer-Cartan element in $(\mc_{rb}^\bullet (A),[\cdot,\cdot]_{rb}) $.
	\end{corollary}
	\begin{proof}
		Since  $\mr\in\mc_{rb}^{1}(A)$ is equivalent to $
		\langle\mr(x),y\rangle=-\langle x,\mr(y)\rangle$ for $x,y\in A$,
		$\mr$ is skew-symmetric with respect to
		$\langle\cdot,\cdot\rangle$ if and only if
		$\mr\in\mc_{rb}^{1}(A)$. Then by Theorem \ref{th-dgla-rb} and
		Theorem \ref{th-dgla-ssrb},  a Maurer-Cartan element $\mr\in
		\mc_{rb}^1(A)$ defines a skew-symmetric Rota-Baxter operator on
		$A$.
		
	\end{proof}
	We next compare the above construction with $\mathcal{C}_{dL}^{\bullet}(V)$ when $A=\en(V)$.
	For every $p\geq 1$, let $\rho_{p} : A^{\otimes p}\rightarrow \en (V^{\otimes p})$ be the canonical isomorphism defined by
	$$
	\rho_{p} (a_1\otimes \cdots \otimes a_p)(v_1\otimes\cdots\otimes v_p)=a_1(v_1)\otimes\cdots\otimes a_p(v_p),
	$$
	for $a_1,\dots ,a_{p} \in A$ and  $v_1,\dots,v_p\in V$.
	For $P\in\en(V^{\otimes p})$, write  $\we P=\rho_p^{-1}(P)\in A^{\otimes p}$. It is obvious that for all $P,Q\in\en (V^{\otimes p})$,  we have $       \we{P\circ Q}=\we{P} \we{Q} $.
	\begin{lemma}
		For all $P\in \en (V^{\otimes m}) $, $m\geq 1$, we have
		$$
		\sigma_{m} P\sigma_{m}^{-1}(v_1,\dots,v_{m})=(\sigma_{m} \we{P} )(v_1,\dots,v_m),\qquad v_1,\dots, v_m\in V.
		$$
		Consequently, $P\in \mathcal{C}_{dL}^p(V)$ if and only if $\sigma_{p+1} \we {P}=(-1)^p \we{P}$.
	\end{lemma}
	\begin{proof}
 
		For $\widetilde{P}=\sum_{\alpha}P_1^{\alpha}\otimes\cdots\otimes P_m^{\alpha} \in A^{\otimes m}$ and $\rho_m(\widetilde{P})=P$, we have
		\begin{align*}
			\sigma_{m}P\sigma_{m}^{-1}(v_1,\dots,v_{m})=&\sigma_{m}(\sum_{\alpha}P_1^{\alpha}(v_2)\otimes\cdots\otimes P_{m-1}^\alpha (v_m)\otimes P_m^{\alpha}(v_1))\\
			=&\sum_{\alpha} P_m^{\alpha}(v_1)\otimes P_1^{\alpha}(v_2)\otimes\cdots\otimes P_{m-1}^\alpha (v_m)\\
			=&(\sigma_{m} \we{P} )(v_1,\dots,v_m),\qquad v_1,\dots, v_m\in V.
		\end{align*}
		Since $P\in \mathcal{C}_{dL}^p(V)$ is equivalent to $\sigma_{p+1} P \sigma_{p+1}^{-1}=(-1)^p P$, we obtain
		$\sigma_{p+1} \we {P}=(-1)^p \we{P}$.
	\end{proof}
	
	For $P\in\en(V^{\otimes(p+1)})$, write
	$\widetilde P=\rho_{p+1}^{-1}(P)
	=\sum_{\alpha}P_0^\alpha\otimes\cdots\otimes P_p^\alpha
	\in A^{\otimes(p+1)}$.   We  define
	\begin{equation}\label{eq-Psi-definition}
		\begin{split}
			\Psi_p(P)(x_1,\ldots,x_p):=\sum_{\alpha}            \langle P_{p-1}^{\alpha},x_1\rangle \langle P_{p-2}^{\alpha},x_2\rangle\cdots           \langle P_0^{\alpha},x_p\rangle P_p^{\alpha},\qquad x_1,\dots,x_p\in A.
		\end{split}
	\end{equation}
	
	For $p=0$, this means $\Psi_0(P)=P$.  The reverse order of the contractions
	in \eqref{eq-Psi-definition} is essential for compatibility with the two
	graded Lie brackets.  Equivalently,
	\begin{equation}\label{eq-Psi-intrinsic}
		\langle\Psi_p( P)(x_1,\ldots,x_p),x_0 \rangle    = \langle \we P,x_p\otimes x_{p-1}\otimes\cdots\otimes x_1\otimes x_0       \rangle.
	\end{equation}
	This intrinsic formula shows in particular that $\Psi_p$ is independent of
	the chosen decomposition of $\widetilde P$ into elementary tensors.
	
	\begin{proposition}\label{prop-dl-rb-cochain-identification}
		For every $p\geq0$, the map $\Psi_p$ restricts to a vector space    isomorphism from $ \mathcal{C}_{dL}^{p}(V)$ to $    \mc_{rb}^{p}(A)$.
		In degree one, if $\we \mu=\sum_i e_i\otimes\mr(e_i^T)$ where $\{e_i\}_{i=1}^n$ is a linear basis of $A$ and  $\{e^T_i\}_{i=1}^n$ is the corresponding dual basis relative to the trace form,
		then $\Psi_1(\mu)=\mr$.
	\end{proposition}
	\begin{proof}
		The nondegeneracy of the trace form  implies that $\Psi_p$ is a vector space isomorphism between $\en (V^{\otimes(p+1)})$ and $\Hom(A^{\otimes p},A)$.

		For any $P\in \mathcal{C}^p_{dL}(V)$ and $ x_0,\dots,x_p\in A$,  we have
		\begin{align*}
			&(-1)^p\langle \Psi_p(P)  (x_1,\ldots,x_p),x_0\rangle
			=(-1)^p\langle \we P,x_p\otimes x_{p-1}\otimes\cdots\otimes x_1\otimes x_0       \rangle\\
			=&\langle \sigma_{p+1} \we P, x_p\otimes \cdots\otimes x_0\rangle
			=\langle \Psi_p(P)(x_0,\dots,x_{p-1}),x_p\rangle,
		\end{align*}
		which means that
		$\sigma_{p+1}P\sigma_{p+1}^{-1}=(-1)^pP$ if and only if $\Psi_p (P)\in \mc_{rb}^p(A)$.
		This  proves that  $\Psi_p$ is an isomorphism between $\mathcal{C}^p_{dL}(V)$ and $\mc_{rb}^p(A)$.  Finally,    since $\{e_i\}$ and $\{e_i^T\}$ are dual bases, $       \Psi_1(\mu)(x)=\sum_i\langle e_i,x\rangle\mr(e_i^T)=\mr(x)$,
		which means that $\Psi_1(\mu)=\mr$.
	\end{proof}

	\begin{theorem}\label{th-dl-rb-dgla-isomorphism}
		Let $P\in\mathcal{C}_{dL}^{p}(V)$ and $Q\in\mathcal{C}_{dL}^{q}(V)$.  Then
		\begin{equation}\label{eq-Psi-bracket-intertwining}
			\Psi_{p+q}\bigl([P,Q]_{dL}\bigr)
			=[\Psi_p(P),\Psi_q(Q)]_{rb}.
		\end{equation}
		Consequently, $         \Psi:(\mathcal{C}_{dL}^{\bullet}(V),[\cdot,\cdot]_{dL})\rightarrow   (\mc_{rb}^{\bullet}(A),[\cdot,\cdot]_{rb}) $
		is an isomorphism of graded Lie algebras.
	\end{theorem}
	\begin{proof}
		Let $F:=\Psi_p(P)$ and $G:=\Psi_q(Q)$.   
		If $p=q=0$, then both brackets are the commutator in $A$.  If $p\geq1$ and $q=0$, then
		\begin{align*}
			&\Psi_p([P,Q]_{dL})(x_1,\ldots,x_p)\\
			=&\Psi_p( \sum_{s=0}^p(-1)^{ps} \sigma_{p+1}^s(P\circ(\id^{\otimes p}\otimes Q) -(Q\otimes \id^{\otimes p})\circ P)\sigma_{p+1}^{-s}  )(x_1,\ldots,x_p)\\
			=&F(x_1,\ldots,x_p)G-GF(x_1,\ldots,x_p)+ \sum_{i=1}^{p}F(x_1,\ldots,x_{i-1},Gx_i-x_iG,x_{i+1},\ldots,x_p) ,
		\end{align*}
		which is $[F,G]_{rb}$ by \eqref{eq-rb-degree-zero-2}.  The case $p=0$
		and $q\geq1$ follows from graded skew-symmetry.
		
		Suppose now that $p,q\geq1$.  Let
		$$
		\we{P}=\sum_{\alpha}P_0^{\alpha}\otimes\cdots\otimes P_p^{\alpha},\qquad
		\we{Q}=\sum_{\beta}Q_0^{\beta}\otimes\cdots\otimes Q_q^{\beta}.
		$$
		Then
		\begin{align*}
			&\we{ P\diamond Q}= \sum_{\alpha,\beta} P_0^\alpha\otimes \cdots \otimes P_{p-1}^\alpha \otimes P_p^\alpha Q_0^\beta\otimes Q_1^\beta \otimes \cdots \otimes Q_q^\beta,\\
			&   \we {Q\diamond P}=\sum_{\alpha,\beta}Q_0^{\beta}\otimes\cdots\otimes Q_{q-1}^\beta\otimes Q_q^{\beta}P_0^\alpha \otimes P_1^\alpha \otimes \cdots\otimes P_p^\alpha.
		\end{align*}
		For the terms related to $P\diamond Q$, we distinguish
		the following three cases.
		\begin{enumerate}
			\item Let $0\leq s\leq q-1$ and  put $j=q-s$. Then
			\begin{align*}
				&(-1)^{(p+q)s}\Psi_{p+q}(\sigma_{p+q+1}^s (P\diamond Q)\sigma_{p+q+1}^{-s}  )(x_1,\dots,x_{p+q})\\
				=&(-1)^{pq+jp}G(x_1,\dots,x_{j-1},x_jF(x_{j+1},\dots,x_{j+p}), x_{j+p+1},\dots,x_{p+q}).
			\end{align*}
			\item   Let $s=q$. Then
			\begin{align*}
				&(-1)^{(p+q)q}\Psi_{p+q}(\sigma_{p+q+1}^q (P\diamond Q)\sigma_{p+q+1}^{-q}  )(x_1,\dots,x_{p+q})
				=(-1)^{pq} F(x_1,\dots,x_p)G(x_{p+1},\dots, x_{p+q}).
			\end{align*}
			\item   Let $q+1\leq s\leq p+q$ and $i=p+q-s+1$. Then
			\begin{align*}
				&(-1)^{(p+q)s}\Psi_{p+q}(\sigma_{p+q+1}^s (P\diamond Q)\sigma_{p+q+1}^{-s}  )(x_1,\dots,x_{p+q})\\
				=&(-1)^{(i-1)q}F(x_1,\dots,x_{i-1},G(x_{i},\dots,x_{i+q-1})x_{i+q}, x_{i+q+1},\dots,x_{p+q}).
			\end{align*}
		\end{enumerate}

		For the terms related to $Q\diamond P$, we distinguish  the following three cases.
		\begin{enumerate}
			\item   Let $0\leq s\leq p-1$ and put $i=p-s$. Then
			\begin{align*}
				&-(-1)^{(p+q)s+pq}\Psi_{p+q}(\sigma_{p+q+1}^s (Q\diamond P)\sigma_{p+q+1}^{-s}  )(x_1,\dots,x_{p+q})\\
				=&-(-1)^{iq}F(x_1,\dots,x_{i-1},x_iG(x_{i+1},\dots,x_{i+q}), x_{i+q+1},\dots,x_{p+q}).
			\end{align*}
			\item  Let $s=p$. Then
			\begin{align*}
				&-(-1)^{p^2}\Psi_{p+q}(\sigma_{p+q+1}^p (Q\diamond P)\sigma_{p+q+1}^{-p}  )(x_1,\dots,x_{p+q})
				=-G(x_1,\dots ,x_q)F(x_{q+1},\dots,x_{p+q}).
			\end{align*}
			\item Let $p+1\leq s\leq p+q$ and $j=p+q-s+1$. Then
			\begin{align*}
				&-(-1)^{(p+q)s+pq}\Psi_{p+q}(\sigma_{p+q+1}^s (Q\diamond P)\sigma_{p+q+1}^{-s}  )(x_1,\dots,x_{p+q})\\
				=&-(-1)^{pq+(j-1)p}G(x_1,\dots,x_{j-1},F(x_{j},\dots,x_{j+p-1})x_{j+p}, x_{j+p+1},\dots,x_{p+q}).
			\end{align*}
			
		\end{enumerate}
		
		Thus, we have
		\begin{align*}
			&\Psi_{p+q}\bigl([P,Q]_{dL}\bigr)(x_1,\ldots,x_{p+q})\\
			=&\sum_{i=1}^{p}(-1)^{(i-1)q}F\bigl(x_1,\ldots,x_{i-1},G(x_i,\ldots,x_{i+q-1})x_{i+q},x_{i+q+1},\ldots,x_{p+q}\bigr)\\
			&-\sum_{i=1}^{p}(-1)^{iq}F\bigl(x_1,\ldots,x_{i-1},x_iG(x_{i+1},\ldots,x_{i+q}),x_{i+q+1},\ldots,x_{p+q}\bigr)\\
			&-(-1)^{pq}\sum_{j=1}^{q}(-1)^{(j-1)p}G\bigl(x_1,\ldots,x_{j-1},F(x_j,\ldots,x_{j+p-1})x_{j+p},x_{j+p+1},\ldots,x_{p+q}\bigr)\\
			&+(-1)^{pq}\sum_{j=1}^{q}(-1)^{jp}G\bigl(x_1,\ldots,x_{j-1},x_jF(x_{j+1},\ldots,x_{j+p}),   x_{j+p+1},\ldots,x_{p+q}\bigr)\\
			&+(-1)^{pq}F(x_1,\ldots,x_p)G(x_{p+1},\ldots,x_{p+q})-G(x_1,\ldots,x_q)F(x_{q+1},\ldots,x_{p+q}).
		\end{align*}
		The right-hand side is precisely
		$F\circledcirc G-(-1)^{pq}G\circledcirc F=[F,G]_{rb}$, which proves \eqref{eq-Psi-bracket-intertwining}.  The
		remaining assertions follow from Proposition
		\ref{prop-dl-rb-cochain-identification} and Theorem \ref{th-dgla}.
	\end{proof}
	Theorem \ref{th-dl-rb-dgla-isomorphism} gives, at the level of Maurer-Cartan elements, a
	cohomological refinement of Theorem \ref{th-dl-rb}.  Indeed, if
	$\mu\in\mathcal{C}_{dL}^{1}(V)$ and $\mr=\Psi_1(\mu)$, then
	\begin{equation}\label{eq-double-jacobi-rb-torsion}
		\Psi_2(\mu\square\mu)(x,y)
		=\mr\bigl(\mr(x)y+x\mr(y)\bigr)-\mr(x)\mr(y).
	\end{equation}
	Thus $\mu\square\mu=0$ if and only if $\mr$ is a Rota-Baxter operator,
	whereas $\mu\in\mathcal{C}_{dL}^{1}(V)$ if and only if $\mr$ is
	skew-symmetric.
	
	Let now $\mu$ be a double Lie bracket and let
	$\mr=\Psi_1(\mu)$.  Twisting the two graded Lie algebras by the corresponding
	Maurer-Cartan elements gives the differentials
	\begin{equation*}
		d_{\mu}P=[\mu,P]_{dL},
		\qquad
		d_{\mr}F=[\mr,F]_{rb}.
	\end{equation*}
	Equation \eqref{eq-Psi-bracket-intertwining} implies $\Psi_{p+1}\circ d_{\mu}=d_{\mr}\circ\Psi_p$ for $p\geq0$.
	
	It follows that if $\mr$ is a skew-symmetric Rota-Baxter operator, then $(\mc_{rb}^\bullet(A),d_\mr=[\mr,\cdot]_{rb})$ is a cochain complex. Its cohomology is called the cohomology of the skew-symmetric Rota-Baxter operator $\mr$ and is denoted by $H_{rb,\mathrm{cyc}}^{\bullet}(A,\mr)= \bigoplus_{p\geq 0} H_{rb,{\rm cyc}}^p(A,\mr)$, where $ H_{rb,{\rm cyc}}^p(A,\mr)$ is defined by
	$$
	H_{rb,\mathrm{cyc}}^p(A,\mr):= \frac{\ker (d_\mr: \mc_{rb }^p(A)\rightarrow \mc_{rb}^{p+1} (A)) } { \mathrm{im} (d_\mr: \mc_{rb}^{p-1}(A)\rightarrow \mc_{rb}^{p} (A))}.
	$$
	We set $\mc_{rb}^{-1}(A):=0$.
	In particular,
	$H_{rb,\mathrm{cyc}}^0(A,\mr)
	=\ker(d_\mr:\mc_{rb}^0(A)\to\mc_{rb}^1(A))$.
	\begin{corollary}\label{cor-dl-rb-cohomology}
		Let $(V,\mu)$ be a finite-dimensional double Lie algebra and let $\mr$ be
		the associated skew-symmetric Rota-Baxter operator on $A=\en(V)$.  Then
		$$
		\Psi:(\mathcal{C}_{dL}^{\bullet}(V),[\cdot,\cdot]_{dL},d_{\mu}) \longrightarrow (\mc_{rb}^{\bullet}(A),[\cdot,\cdot]_{rb},d_{\mr})
		$$
		is an isomorphism of dg Lie algebras.  Consequently,  $H_{dL}^\bullet(V,\mu)\cong H_{rb,\mathrm{cyc}}^{\bullet}(A,\mr)$.
	\end{corollary}

	\begin{remark}
		The cyclic Rota-Baxter cochain complex $(\mc_{rb}^{\bullet}(A),d_\mr )$ is generally a proper subcomplex of
		the full Rota-Baxter deformation complex   $(\bigoplus_{p\geq0}\Hom(A^{\otimes p},A),d_\mr)$.
		In the matrix case, Corollary \ref{cor-dl-rb-cohomology} identifies this cyclic Rota-Baxter cochain complex with
		the double Lie algebra cochain complex.
	\end{remark}

	\subsection{A cyclic-cohomological interpretation}
	\begin{lemma}\label{lem-M}
		Let $A$ be an associative algebra and $\mr$ be a Rota-Baxter operator. Define  left and right actions of $A_{\mr}$ by
		\begin{align}
			&l_\mr:A_\mr\rightarrow \en(A),\qquad   l_\mr(a)b:=\mr(a)b-\mr(ab),                                      \label{eq-RB-left-module}\\
			&r_\mr:A_\mr\rightarrow \en(A),\qquad  r_\mr(a)b:=b\mr(a)-\mr(ba),                            \qquad a,b\in A.          \label{eq-RB-right-module}
		\end{align}
		Then  $(l_\mr,r_\mr)$ defines an $A_\mr$-bimodule structure on $A$. We denote this bimodule by $M={}_{l_\mr}A_{r_\mr}$.
	\end{lemma}
	\begin{proof}
		Let $a,b,c \in A$.   We obtain
		\begin{align*}
			l_\mr(a\cdot_\mr b)  c=&\mr(a \cdot_\mr b) c-\mr( (a\cdot_\mr b)c)=\mr(a)\mr(b)c-\mr(\mr(a)bc+a\mr(b)c)\\
			=& \mr(a)\mr(b)c-\mr(a)\mr(bc)  -\mr\bigl(a\mr(b)c-a\mr(bc)\bigr)=l_\mr (a) (l_\mr(b) c).
		\end{align*}
		
		Similarly, we can prove that
		$$
		r_\mr( c ) (r_\mr(b) a)=r_\mr (b\cdot_\mr c) a,\qquad
		l_\mr(a)(r_\mr (c)b)=r_\mr(c)(l_\mr(a)b).
		$$
		Therefore, $(l_\mr,r_\mr)$ defines an $A_\mr$-bimodule structure on $A$.
	\end{proof}
	For $p\geq0$, set $ \mc^p(A_{\mr},M):=\Hom(A_{\mr}^{\otimes p},M) =\Hom(A^{\otimes p},A)$.
	The Hochschild differential
	$ d_H:\mc^p(A_{\mr},M)\longrightarrow \mc^{p+1}(A_{\mr},M)$
	is given by
	\begin{align*}
		&(d_H F)(x_1,\ldots,x_{p+1}) \notag\\
		=&l_\mr(x_1)F(x_2,\ldots,x_{p+1})+\sum_{i=1}^{p}(-1)^i  F(x_1,\ldots,x_i\cdot_{\mr}x_{i+1},\ldots,x_{p+1})+(-1)^{p+1}   r_\mr(x_{p+1})F(x_1,\ldots,x_p).
	\end{align*}
	For $p=0$, this becomes $ (d_H a)(x)=l_\mr(x)a-r_\mr(x)a$.  The cohomology of $(\mc^\bullet(A_{\mr},M)=\bigoplus_{p\geq 0}\mc^p(A_{\mr},M), d_H )$ is denoted by $HH^\bullet(A_{\mr},M)=\bigoplus_{p\geq 0}HH^p(A_{\mr},M) $.
	\begin{proposition}[{ \cite{Das20}}]\label{prop-RB-Hochschild-complex}
		With the sign conventions adopted in the present paper, for every
		$F\in\Hom(A^{\otimes p},A)$, one has
		\begin{equation}\label{eq-RB-Hochschild-sign}
			d_{\mr}F=[\mr,F]_{rb}= (-1)^p d_H F.
		\end{equation}
		Consequently, the map
		$   S_p(F):=(-1)^{\frac{p(p-1)}2}F$
		defines an isomorphism of cochain complexes
		$$
		S:
		\bigl(\mc^\bullet(A_\mr,M),d_\mr\bigr)
		\xrightarrow{\ \cong\ }
		\bigl(\mc^\bullet(A_\mr,M),d_H\bigr).
		$$
	\end{proposition}

	\begin{lemma}\label{lem-RB-bimodule}
		Let $(A,\langle\cdot,\cdot\rangle)$ be a finite-dimensional symmetric Frobenius algebra  and $ \mr$ be a skew-symmetric Rota-Baxter operator on $(A,\langle\cdot,\cdot\rangle)$.
		Define  left and right actions of $A_{\mr}$ on ${A_\mr}^*$ by
		\begin{align*}
			(a\cdot f)(x):=f(x\cdot_\mr a),\qquad (f\cdot a)(x):=f(a\cdot_\mr x),\qquad a,x\in A_\mr, f\in {A_\mr}^*.
		\end{align*}
		Define $\iota : M\rightarrow {A_\mr}^*$, $\iota(x) (a):=\langle x,a\rangle$ for $x\in M$ and $a\in A_\mr$.   Then $\iota$ is an $A_\mr$-bimodule isomorphism.
	\end{lemma}
	\begin{proof}
		Since $\langle\cdot,\cdot\rangle$ is non-degenerate, $\iota$ is a linear isomorphism.
		For $a,b\in A_\mr $, $x\in M$, we have
		\begin{align*}
			\langle l_\mr (a) x,b\rangle=& \langle \mr(a)x- \mr(ax),b\rangle =\langle x, b\mr (a)\rangle +\langle ax,\mr(b)\rangle\\
			=&\langle x, b\cdot_\mr a\rangle= (a\cdot \iota(x) )(b),
		\end{align*}
		which means that $ \iota( l_\mr (a) x)=a\cdot \iota (x)$. Similarly, we can prove that $\iota(r_\mr(a)x)=\iota(x)\cdot a$.
		Thus, $\iota$ is an $A_\mr$-bimodule isomorphism.
	\end{proof}

	We now recall some notions about cyclic  cohomology
	\cite{Connes85,Loday98}.  Put
	$\mc^p_{\mathrm{Hoch}}(A_{\mr}):=\Hom(A_{\mr}^{\otimes(p+1)},\bfk)$.
	The scalar-valued Hochschild differential $d_C:
	\mc^p_{\mathrm{Hoch}}(A_{\mr})\rightarrow
	\mc^{p+1}_{\mathrm{Hoch}}(A_{\mr})$  is given by ($F\in
	\Hom(A_{\mr}^{\otimes(p+1)},\bfk)$)
	\begin{align*}
		&(d_C F)(x_0,\ldots,x_{p+1})\\
		=&\sum_{i=0}^{p}(-1)^i
		F(x_0,\ldots,x_i\cdot_{\mr}x_{i+1},\ldots,x_{p+1})+(-1)^{p+1}   F(x_{p+1}\cdot_{\mr}x_0,x_1,\ldots,x_p).
	\end{align*}
	
	Define the signed cyclic operator by $ (\tau_p F)(x_0,\ldots,x_p):=(-1)^p F(x_p,x_0,\ldots,x_{p-1})$,
	and set
	\begin{equation*}
		\mc^p_{\tau}(A_{\mr}):=\ker(1-\tau_p)   \subseteq \mc^p_{\mathrm{Hoch}}(A_{\mr}).
	\end{equation*}
	The standard identity $ (1-\tau_{p+1})d_C=d_C'(1-\tau_p)$, where $d_C'$ is defined by
	$ (d_C' F)(x_0,\dots,x_{p+1}) = \sum_{i=0}^{p} (-1)^i F(x_0,\dots,x_i \cdot_\mr x_{i+1},\dots,x_{p+1})$, shows that $ d_C\bigl(\mc^p_{\tau}(A_{\mr})\bigr) \subseteq \mc^{p+1}_{\tau}(A_{\mr}) $. Thus $(\mc_\tau^\bullet (A_\mr)=\bigoplus_{p\geq0} \mc_\tau^p(A_\mr), d_C)$ is a \textbf{cyclic complex} of $A_\mr$.
	We define the \textbf{cyclic  cohomology} $HC^\bullet_\tau(A_\mr)= \bigoplus_{p\geq 0}HC^p_{\tau}(A_{\mr})$, where
	$$
	HC^p_{\tau}(A_{\mr}):=H^p (\mc^\bullet_{\tau}(A_{\mr}),d_C ).
	$$

	\begin{proposition}\label{prop-cyclic-RB-cyclic-Hochschild}
		Let $A$ and $A_\mr$ be as in Lemma \ref{lem-RB-bimodule}.
		For $p\geq0$, define $J_p:\Hom(A^{\otimes p},A)\rightarrow      \Hom(A^{\otimes(p+1)},\bfk)$ by
		\begin{equation*}
			J_p(F)(x_0,\ldots,x_p):=\langle F(x_1,\ldots,x_p),x_0\rangle.
		\end{equation*}
		Then $J_p$ is a vector space isomorphism and satisfies $    J_{p+1}\circ d_{H}=d_C\circ J_p$.
		Consequently,     $J_{p+1}\circ d_{\mr}=(-1)^p d_C\circ J_p $. Moreover, $J_p$ restricts to an isomorphism $J_p: \mc_{rb}^{p}(A) \rightarrow \mc^p_{\tau}(A_{\mr})$. For $p\geq0$, set
		$  \widetilde J_p :=  (-1)^{\frac{p(p-1)}2}J_p$.
		Then
		$  \widetilde J: (\mc_{rb}^{\bullet}(A),d_\mr\bigr)
		\xrightarrow{\ \cong\ }
		\bigl(\mc_{\tau}^{\bullet}(A_\mr),d_C)$
		is an isomorphism of cochain complexes. Consequently,
		$$
		H_{rb,\mathrm{cyc}}^{p}(A,\mr) \cong  HC_{\tau}^{p}(A_\mr), \qquad p\geq0.
		$$
	\end{proposition}
	\begin{proof}
		The map $J_p$ is the composite
		$$
		\Hom({A_\mr}^{\otimes p},M)     \xrightarrow{\ \iota_*\ }   \Hom({A_\mr}^{\otimes p},{A_\mr}^*)     \cong   \Hom({A_\mr}^{\otimes(p+1)},\bfk),
		$$
		and is therefore a vector space isomorphism.
		
		For $F\in\Hom({A_\mr}^{\otimes p},M)$, we have
		\begin{align*}
			&J_{p+1}(d_HF)(x_0,\ldots,x_{p+1})\\
			=&\langle l_\mr(x_1)F(x_2,\ldots,x_{p+1}),x_0\rangle
			+   \sum_{i=1}^{p}(-1)^i\langle F(x_1,\ldots,x_i\cdot_\mr x_{i+1},\ldots,x_{p+1}),x_0\rangle\\
			&+  (-1)^{p+1}\langle       r_\mr(x_{p+1})F(x_1,\ldots,x_p),x_0\rangle.
		\end{align*}
		By Lemma \ref{lem-RB-bimodule}, this is equal to
		\begin{align*}
			&J_p(F)(x_0\cdot_\mr x_1,x_2,\ldots,x_{p+1})
			+\sum_{i=1}^{p}(-1)^iJ_p(F)(x_0,\ldots,x_i\cdot_\mr x_{i+1},\ldots,x_{p+1})\\
			&+  (-1)^{p+1}J_p(F)(x_{p+1}\cdot_\mr x_0,x_1,\ldots,x_p),
		\end{align*}
		which is precisely
		$   (d_C J_p(F))(x_0,\ldots,x_{p+1})$. Thus,  $ J_{p+1}\circ d_{H}=d_C\circ J_p$ and then $J_{p+1}\circ d_{\mr}=(-1)^p d_C\circ J_p $ by    $d_\mr=(-1)^p d_H$.
		
		Consequently, $      J_p(F)(x_0,\ldots,x_p)= \langle F(x_1,\ldots,x_p),x_0\rangle$,
		and hence $J_p(F)$ is cyclic if and only if
		\begin{align*}
			\langle F(x_1,\ldots,x_p),x_0\rangle =      (-1)^p  \langle F(x_0,\ldots,x_{p-1}),x_p\rangle.
		\end{align*}
		This is exactly the defining condition for  $F\in\mc_{rb}^p(A)$.
		Finally,
		\begin{align*}
			\widetilde J_{p+1}\circ d_\mr   &=      (-1)^\frac{p(p+1)}{2}J_{p+1}\circ d_\mr =(-1)^{\frac{p(p+1)}{2}}(-1)^p  d_C\circ J_p\\
			&= (-1)^{\frac{p(p-1)}{2}} d_C\circ J_p=d_C\circ\widetilde J_p.
		\end{align*}
		Therefore, $\widetilde J$ is an isomorphism of cochain complexes.
		Passing to cohomology gives
		\begin{equation*}
			H_{rb,\mathrm{cyc}}^{p}(A,\mr)\cong HC_{\tau}^{p}(A_\mr),\qquad p\geq0.\qedhere
		\end{equation*}
	\end{proof}

	Recall  from Theorem \ref{th-dl-rb-dgla-isomorphism} that $ \Psi_p:\mathcal{C}_{dL}^{p}(V) \xrightarrow{\ \cong\ }\mc_{rb}^{p}(A)$
	and $ \Psi_{p+1}\circ d_\mu=d_{\mr}\circ\Psi_p$.
	The sign in \eqref{eq-RB-Hochschild-sign} must be taken into account in
	order to obtain an actual isomorphism of cochain complexes.
	
	\begin{theorem}\label{th-dl-cyclic-Hochschild}
		Let $(V,\mu)$ be a finite-dimensional double Lie algebra. Let $\mr=\Psi_1(\mu)$ be the associated skew-symmetric Rota-Baxter operator on $A=\en(V)$.
		For $p\geq0$, define $  \Phi_p:\mathcal{C}_{dL}^{p}(V)\longrightarrow \mc^p_{\tau}(A_{\mr})$ by
		\begin{align*}
			\Phi_p(P)(x_0,\ldots,x_p) :=(-1)^{\frac{p(p-1)}{2}}\langle          \Psi_p(P)(x_1,\ldots,x_p),x_0        \rangle.
		\end{align*}
		Then
		$\Phi:  (\mathcal{C}_{dL}^{\bullet}(V),d_\mu )\rightarrow(\mc^\bullet_{\tau}(A_{\mr}),d_C )$
		is an isomorphism of cochain complexes.  Consequently,
		\begin{equation*}
			H^p_{dL}(V,\mu) \cong HC^p_{\tau}(A_{\mr}), \qquad p\geq0.
		\end{equation*}
	\end{theorem}
	
	\begin{proof}
		Since $\Psi_p$ and $J_p$ are vector space isomorphisms and
		$   J_p:\mc_{rb}^p(A)\xrightarrow{\cong}\mc_{\tau}^p(A_\mr)$,
		each $\Phi_p=(-1)^{\frac{p(p-1)}2}J_p\Psi_p$ is a vector space isomorphism.
		
		Moreover, since
		\begin{align*}
			d_C\circ\Phi_p&=    (-1)^{\frac{p(p-1)}2} d_C \circ J_p\circ \Psi_p =(-1)^{\frac{p(p-1)}2}(-1)^pJ_{p+1}\circ d_\mr\circ \Psi_p
			=(-1)^{\frac{p(p+1)}2}J_{p+1}\circ \Psi_{p+1}\circ d_\mu,
		\end{align*}
		we have $   d_C\circ \Phi_p =   \Phi_{p+1}\circ d_\mu$.
		Thus $\Phi$ is an isomorphism of cochain complexes.
	\end{proof}
	\begin{remark}\label{rem-Hochschild-versus-cyclic}
		For an arbitrary Rota-Baxter operator $\mr$ on an associative algebra $A$, Das  \cite{Das20} identified the Rota-Baxter cochain complex with the Hochschild complex of the descendent associative algebra $A_\mr$ with coefficients in the $A_\mr$-bimodule $M={}_{l_{\mr}}A_{r_{\mr}}$ defined in Lemma \ref{lem-M}.
		However, the complex considered here corresponds to its cyclic subcomplex. For example, let $V\neq 0$, $A=\en (V)$, $\dim (A)=m$ and $\mr=0$, so that the corresponding double Lie bracket is $\mu =0$. 
		Therefore,
		$$
		\dim HH^1(A_{\mr},A_{\mr}^*)=m^2,\qquad 
		\dim H^1_{dL}(V,\mu)
		=\dim HC^1_{\tau}(A_{\mr})
		=\frac{m(m-1)}2.
		$$
		Hence the double Lie algebra cohomology is not, in general, the full Hochschild cohomology   $HH^\bullet(A_{\mr},A_{\mr}^*)$.

	\end{remark}

	\subsection{Application: a three-dimensional example}\label{ex-3d-cyclic}
	Let $V=\bfk x_1\oplus \bfk x_2\oplus \bfk x_3$ be a 3-dimensional double Lie algebra whose nonzero values of double Lie bracket $\lc\cdot,\cdot\rc$ are given by
	\begin{equation*}
		\lc x_2,x_2\rc=x_2\otimes x_1-x_1\otimes x_2,
		\qquad
		\lc x_3,x_3\rc=x_3\otimes x_1-x_1\otimes x_3.
	\end{equation*}
	Let   $E_{ij}\in A=\en(V)$ be the matrix unit defined by
	$E_{ij}(x_k)=\delta_{jk}x_i$. Then $  \langle E_{ij},E_{kl}\rangle
	=  \operatorname{tr}(E_{ij}E_{kl})=    \delta_{jk}\delta_{il}$, so
	the basis trace-dual to $\{E_{ij}\}$ is $\{E_{ji}\}$. By
	\eqref{eq-dl-rb}, there is a skew-symmetric Rota-Baxter operator
	$\mr$ on $\en(V)$ whose nonzero actions are given by
	\begin{equation*}
		\mr(E_{21})=-E_{22},        \qquad          \mr(E_{22})=E_{12},         \qquad
		\mr(E_{31})=-E_{33},        \qquad          \mr(E_{33})=E_{13}.
	\end{equation*}

	We now compute the low-degree cohomology of this double Lie
	algebra by using the cyclic cohomology of the descendent
	associative algebra $A_\mr$.   The nonzero products of matrix
	units in $A_\mr$ are
	\begin{align*}
		&E_{11}\cdot_\mr E_{22}=E_{22}\cdot_{\mr}E_{22}=E_{33}\cdot_{\mr}E_{32}=-E_{12}\cdot_{\mr}E_{21}=E_{12},\\
		& E_{11}\cdot_\mr E_{33}=-E_{13}\cdot_{\mr}E_{31}=E_{22}\cdot_{\mr}E_{23}=E_{33}\cdot_{\mr}E_{33}=E_{13},\\
		&E_{21}\cdot_\mr E_{21}=-E_{21},\qquad   E_{21}\cdot_{\mr}E_{33} =-E_{23}\cdot_{\mr}E_{31}=  -E_{21}\cdot_\mr E_{23}=E_{23},\\
		&E_{31}\cdot_{\mr}E_{22}=-E_{31}\cdot_{\mr}E_{32}=-E_{32}\cdot_{\mr}E_{21}=E_{32},\qquad E_{31}\cdot_{\mr}E_{31}=-E_{31},\\
		&E_{22}\cdot_{\mr}E_{21}=E_{11}-E_{22},\qquad E_{33}\cdot_{\mr}E_{31}=E_{11}-E_{33}.
	\end{align*}
	When considering scalar-valued cyclic cochains, let
	$E_{ij}^*\in A^*$ denote the coordinate functional determined by $E_{ij}^*(E_{kl})=\delta_{ik}\delta_{jl}$.
	For $f\in\mc_{\tau}^0(A_\mr)=A^*$, one has $(d_Cf)(X,Y)=f(X\cdot_\mr Y-Y\cdot_\mr X)$.
	It follows from the above multiplication table that
	\begin{equation*}
		[A_\mr,A_\mr]
		=\operatorname{span}\left\{
		E_{12},E_{13},E_{23},E_{32},
		E_{22}-E_{11},E_{22}-E_{33}
		\right\}.
	\end{equation*}
	In particular, $\dim[A_\mr,A_\mr]=6$, and therefore
	\begin{align*}
		HC_{\tau}^0(A_\mr)   =\ker (d_C:\mc_{\tau}^0(A_\mr) \longrightarrow\mc_{\tau}^1(A_\mr) )
		=\operatorname{span}\{  E_{21}^*,E_{31}^*,  E_{11}^*+E_{22}^*+E_{33}^*\}    \cong\bfk^3.
	\end{align*}
	
	\begin{proposition}
		For the cyclic cochain complex of $A_\mr$, one has $\dim HC_\tau^1(A_\mr)=1$.
	\end{proposition}
	\begin{proof}
		Since $F(x,y)=-F(y,x)$ for  $F\in \mc_{\tau}^1(A_\mr)$ and $x,y\in A_\mr$,
		$\mc_\tau^1(A_\mr)=\bigwedge^2{A_\mr}^*$ and hence $\dim \mc_\tau^1(A_\mr)=\binom92=36$.
		
		Thus, we obtain
		$Z_{\tau}^1(A_\mr):=\ker\bigl(d_C:\mc_{\tau}^1(A_\mr)\longrightarrow\mc_{\tau}^2(A_\mr)\bigr)=\operatorname{span}\{\zeta_1,\ldots,\zeta_7\}$,
		where
		\begin{align*}
			\zeta_1= &E_{21}^*\wedge E_{22}^*,\qquad    \zeta_2=  E_{11}^*\wedge E_{33}^*-E_{13}^*\wedge E_{31}^*+E_{22}^*\wedge E_{23}^*,\\
			\zeta_3=&E_{21}^*\wedge E_{23}^*-E_{21}^*\wedge E_{33}^*+E_{23}^*\wedge E_{31}^*,\\
			\zeta_4=&E_{11}^*\wedge E_{22}^*-E_{11}^*\wedge E_{33}^*-E_{12}^*\wedge E_{21}^*+E_{13}^*\wedge E_{31}^* -E_{22}^*\wedge E_{33}^*
			+E_{23}^*\wedge E_{32}^*,\\
			\zeta_5=&-E_{21}^*\wedge E_{32}^*+E_{22}^*\wedge E_{31}^*+E_{31}^*\wedge E_{32}^*,\\
			\zeta_6=&E_{31}^*\wedge E_{33}^*,\qquad
			\zeta_7=-E_{11}^*\wedge E_{22}^*+E_{12}^*\wedge E_{21}^*+E_{32}^*\wedge E_{33}^*.
		\end{align*}
		Thus, $\dim Z_{\tau}^1(A_\mr)=7$.  On the other hand, the
		nonzero values of $d_C$ on the coordinate functionals are
		\begin{align*}
			d_CE_{11}^*=&-E_{21}^*\wedge E_{22}^*-E_{31}^*\wedge E_{33}^*,\qquad
			d_CE_{12}^*=E_{11}^*\wedge E_{22}^*-E_{12}^*\wedge E_{21}^*-E_{32}^*\wedge E_{33}^*,\\
			d_CE_{13}^*=&E_{11}^*\wedge E_{33}^*-E_{13}^*\wedge E_{31}^*+E_{22}^*\wedge E_{23}^*,\qquad
			d_CE_{22}^*=E_{21}^*\wedge E_{22}^*,\\
			d_CE_{23}^* =&-E_{21}^*\wedge E_{23}^*+E_{21}^*\wedge E_{33}^*-E_{23}^*\wedge E_{31}^*,\qquad
			d_CE_{32}^*= E_{21}^*\wedge E_{32}^*-E_{22}^*\wedge E_{31}^*-E_{31}^*\wedge E_{32}^*,\\
			d_CE_{33}^*=&E_{31}^*\wedge E_{33}^*.
		\end{align*}
		Consequently, $  B_{\tau}^1(A_\mr):=\operatorname{im}\bigl(d_C:\mc_{\tau}^0(A_\mr)\longrightarrow\mc_{\tau}^1(A_\mr)\bigr)=\operatorname{span}\{\zeta_1,\zeta_2,\zeta_3,\zeta_5,\zeta_6,\zeta_7\}$,
		and $\dim B_{\tau}^1(A_\mr)=6$.  It follows that
		$$
		HC_{\tau}^1(A_\mr)
		=\frac{Z_{\tau}^1(A_\mr)}{B_{\tau}^1(A_\mr)}
		\cong\bfk.
		$$
		Hence $\dim HC_\tau^1(A_\mr)=1$.
	\end{proof}
	A convenient representative of its nonzero class is $\eta   =(E_{22}^*-E_{32}^*)\wedge(E_{23}^*-E_{33}^*)$.
	Indeed, $\eta=\zeta_4+\zeta_2+\zeta_7$, so
	$\eta\in Z_{\tau}^1(A_\mr)$, whereas
	$\eta\notin B_{\tau}^1(A_\mr)$.  Hence, $   HC_{\tau}^1(A_\mr)=\bfk[\eta]$.
	
	By Theorem \ref{th-dl-cyclic-Hochschild}, the cyclic
	cochain complex of $A_\mr$ is isomorphic to the double Lie
	algebra cochain complex.  Therefore,
	$$
	H_{dL}^0(V,\mu)\cong HC_{\tau}^0(A_\mr)\cong\bfk^3,
	\qquad
	H_{dL}^1(V,\mu)\cong HC_{\tau}^1(A_\mr)\cong\bfk.
	$$
	More precisely, under the cochain isomorphism $\Phi$ of
	Theorem \ref{th-dl-cyclic-Hochschild}, the generator of $H_{dL}^1(V,\mu)$ is represented by
	$\nu=\Phi_1^{-1}(\eta)$.
	
	Let  $ F:=\Psi_1(\nu)$. Since $d_C\eta=0$ and $\Phi$ is an isomorphism of cochain complexes, we have
	$d_\mu\nu=[\mu,\nu]_{dL}=0$.
	Equivalently, the dg Lie algebra isomorphism
	$\Psi$ gives $d_\mr F=[\mr,F]_{rb}=0$.
	Since $\Phi_1=J_1\Psi_1$, the map $F$ is determined by $\langle F(X),Y\rangle=\eta(Y,X)$ for $X,Y\in A$.
	Its nonzero values are
	\begin{align*}
		F(E_{22})=E_{33}-E_{32},\qquad  F(E_{23})=E_{22}-E_{23},\qquad
		F(E_{32})=E_{32}-E_{33},\qquad  F(E_{33})=E_{23}-E_{22}.
	\end{align*}

	\noindent {\bf Acknowledgments.}
	This work is supported by NSFC (12271265, 12261131498, W2412041,
	12671037), Fundamental Research Funds for the Central Universities
	and Nankai Zhide Foundation.
	
	\smallskip
	
	\noindent
	{\bf Declaration of interests. } The authors have no conflicts of interest to disclose.
	
	\smallskip
	
	\noindent
	{\bf Data availability. } No new data were created or analyzed in this study.

\end{document}